\documentclass{amsart}
\usepackage[utf8]{inputenc}
\usepackage{graphicx,graphics,hyperref}
\usepackage{caption}
\usepackage{subfig}
\usepackage{amsmath,amssymb}
\usepackage{tikz}
\usetikzlibrary{arrows.meta}

\usepackage{epstopdf}
\usepackage{bbm}
 
\usepackage{enumerate}
\usepackage{mathtools}
\usepackage{color,xcolor} 
\usepackage{mathrsfs,bbm} 
\DeclareMathAlphabet\mathbfcal{OMS}{cmsy}{b}{n}

\newcommand{\norm}[1]{\left\lVert#1\right\rVert}
\newcommand{\abs}[1]{\left|#1\right|}
\newcommand{\boldsymb}{ }

\renewcommand{\H}{ H}

\renewcommand{\Re}{\mathrm{Re}\,}

\newtheorem{lemma}{Lemma}
\newtheorem{remark}{Remark}
\newtheorem{proposition}{Proposition}
\newtheorem{theorem}{Theorem}

\newcommand{\tot}{^\text{tot}}

\newcommand{\inc}{^\text{inc}}

\newcommand{\cqb}{\boldsymb b}
\newcommand{\cqg}{\boldsymb g}
\newcommand{\cqC}{\boldsymb C}
\newcommand{\cqK}{\boldsymb K}
\newcommand{\cqH}{\boldsymb H}
\newcommand{\cqL}{\boldsymb L}
\newcommand{\cqW}{\boldsymb W}
\newcommand{\cqX}{\boldsymb V}
\newcommand{\cqY}{\boldsymb W}
\newcommand{\vtn}{\underline{t_n}}
\newcommand{\bone}{\mathbbm{1}}

\newcommand{\wu}{\widehat{u}}
\newcommand{\wg}{\widehat{g}}
\newcommand{\Obs}{\mathcal O}

\title[A posteriori estimates and mesh refinements for scattering problems]{A posteriori error estimates and space-adaptive mesh refinements for time-dependent scattering problems}

\author{Th\'{e}ophile Chaumont-Frelet}
\address{Inria Univ.~Lille and Laboratoire Paul Painlev\'{e}, 59655 Villeneuve-d’Ascq, France.} 
\email{theophile.chaumont@inria.fr}
\author{Heiko Gimperlein}
\address{Engineering Mathematics, University of Innsbruck, 6020 Innsbruck,  Austria.}
\email{heiko.gimperlein@uibk.ac.at}
\author{Ignacio Labarca-Figueroa}
\address{Institute for Mathematical and Computational Engineering, School of Engineering and Faculty of Mathematics, Pontificia Universidad Cat\'{o}lica de Chile, Santiago, Chile, and \newline\indent Engineering Mathematics, University of Innsbruck, 6020 Innsbruck,  Austria.}
\email{ignacio.labarca@uc.cl}
\author{ Jörg Nick}
\address{ETH Zürich,	Seminar for Applied Mathematics,
	 Rämistrasse 101, CH-8092 Zürich, Switzerland. %
}
\email{joerg.nick@math.ethz.ch}
\begin{document}
\maketitle
\begin{abstract}
This work studies a posteriori error estimates and their use for time-dependent acoustic scattering problems, formulated as a time-dependent boundary integral equation based on a single-layer ansatz. The integral equation is discretized by the convolution quadrature method in time and by boundary elements in space. We prove the reliability of an error estimator of residual type and study the resulting space-adaptive mesh refinements. Moreover, we present a simple modification of the convolution quadrature method based on temporal shifts, which recovers, for the boundary densities, the full classical temporal convergence order $2m-1$ of the temporal convolution quadrature method based on the $m$-stage convolution quadrature semi-discretization. We numerically observe that the adaptive scheme yields asymptotically optimal meshes for an acoustic scattering problem in two dimensions.
\end{abstract}

\section{Introduction and problem formulation}

The computational treatment of time-domain scattering problems by formulations based on boundary integral equations is by now a classical approach. Provably convergent numerical methods of those formulations mostly fall into two categories and either discretize the time-domain fundamental solution via a Galerkin approach (referred to as space-time Galerkin), or rely on underlying time stepping schemes to approximate inverse Laplace transforms, which leads to convolution quadrature methods. Galerkin schemes with adaptive local mesh refinements in space, in time or in space-time have been investigated recently, e.g., in \cite{aimiw,aimispacetime,goss20,Glaefke,hoonhout2023,SV16,zank2020inf,zank1d}. Beyond boundary integral equations, similar spatially adaptive methods for the wave equation based on a finite element discretization in space and time stepping schemes or the time-continuous problem in time have been considered in  \cite{cf23,chaumont2025damped}.
In this paper, we initiate the study of a posteriori error estimates and the resulting adaptive mesh refinement procedures based on convolution quadrature methods. For these methods, Laplace domain operators have to be approximated at a large number of wavenumbers. The computational bottleneck of these schemes is the assembly and storage of the discretized Laplace domain boundary integral operators, most commonly in the form of boundary element matrices. 

Different approaches have been considered to reduce the number of necessary frequencies at which these matrices have to be approximated. Convolution quadrature schemes that use variable time step sizes have been introduced in \cite{LS13,LS15,ss} and were further used to derive a temporally adaptive scheme in \cite{SV16}. Other approaches to reduce the computational cost in the context of wave-types problems derive higher-order methods \cite{BLM11,BF24,LS16} or even or even $p$-versions \cite{rieder25}, or alternatively reduce the number of necessary evaluations through specialized quadrature for the approximation of highly oscillatory integrals \cite{BLS17,MN24}. These variants of convolution quadrature methods reduce the number of frequencies at which the Laplace domain operators have to be approximated, often at the price of some computational overhead. 

Unlike for space-time Galerkin methods, improved spatial discretizations, however, have attracted less attention. Efforts on improving the efficiency include compression schemes \cite{BK14,BLS21} or wavelet based discretizations \cite{DF20}. Based on the analysis by Plamenevskii and collaborators \cite{korikov2021asymptotic}, it is known that for scattering problems in polygonal domains quasi-optimal approximations of the solution are achieved by time-independent graded meshes in space \cite{muller2015}. These results were later extended to further problems with geometric singularities and to approximations by $hp$ methods \cite{aimi2023,bansal2021,graded,hp,luong2015,muller2016}.

\subsection*{Our contribution and outline.}
Reducing the number of necessary spatial degrees of freedom to compute approximations of a desired accuracy is therefore a key opportunity to decrease both the runtime and the memory requirements of fully discrete schemes based on the combination of the convolution quadrature method with boundary elements. In this paper, we provide an a posteriori error analysis of the error introduced by the spatial boundary element discretization applied to time-dependent and time-discrete boundary integral equations. We derive error indicators that are, up to a data oscillation term, equivalent to the error. These bounds are formulated in Theorems~\ref{th:full}--\ref{thm:full-no-derivatives}, the main theoretical results in this article. Based on these error indicators, we derive a spatially adaptive scheme that is numerically shown to produce grids that recover the optimal rate of convergence with respect to the spatial degrees of freedom. This numerical evidence indicates that the proposed procedure produces quasi-optimal time-independent meshes \cite{graded,muller2015}, although a proof of quasi-optimality of the generated meshes is beyond the scope of this paper and left for future research.

\subsection{Problem formulation}
  Let $\Omega\subset \mathbb R^d$, $d=2,3,$ be the exterior of a bounded Lipschitz domain or of a Lipschitz screen. On this exterior domain, we impose the acoustic wave equation
\begin{align}\label{eq:acoustic}
	\partial_t^2 u\tot - \Delta u\tot  = 0, \quad \text{in} \quad \Omega^+.
\end{align}
Homogeneous Dirichlet boundary conditions are imposed along the boundary of the scatterers $\Omega$,
\begin{align}
	u\tot = 0,	\quad \text{on} \quad \Gamma = \partial \Omega .
\end{align}
Initially, the total wave is given by an incoming wave $u\inc$ with support away from the boundaries of the scatterers.  The objective of the scattering problem is then to compute the unknown scattered wave $u$, which initially vanishes, such that the total wave $u\tot = u+u\inc$ fulfills the above boundary value problem.

\section{The Laplace domain problem}
\subsection{Boundary integral formulation in the Laplace domain}
This section describes the setting of the Laplace domain problem, for some fixed frequency $s\in\mathbb C$ with positive real part. 

Applying the Laplace transform to the acoustic wave equation reveals the Helmholtz problem, which reads for $s\in\mathbb C$ with $\Re s >0$
\begin{alignat}{2}\label{eq:Helmholtz}
	s^2 \wu - \Delta \wu  &= 0, \quad &&\text{in} \quad \mathbb R^d \setminus \Obs,
	\\ \wu &= -\wg\inc,	\quad &&\text{on} \quad \Gamma= \partial \Obs.
\end{alignat}
For $|x|\rightarrow \infty$, an asymptotic condition is further imposed, that ensures the square integrability of $\wu$. 
The fundamental solutions of the Helmholtz problem in two  and three dimensions are given by 
\begin{align*}
    G(s,x) = \dfrac{i}{4} H_0^{(1)}(isr), \quad x \in \mathbb R^2,\quad \text{ resp. } \quad  G(s,x)=\dfrac{e^{-s|x|}}{4\pi|x|},\quad x \in \mathbb R^3,
\end{align*}
in terms of the Hankel function $H_0^{(1)}$. We reformulate this problem as a boundary integral equation, fully formulated on the boundary of the scattering obstacle $\partial \Obs$.
The potential operator reads
\begin{align*}
	S(s)\varphi (x) = \int_{\Gamma} G(s,x-y) \varphi(y), \quad x\in \mathbb R^d\setminus \Obs.
\end{align*}
Taking the trace towards the boundary $x\rightarrow \Gamma$, which gives the single-layer boundary operator, which reads
\begin{align*}
	V(s)\varphi (x) = \int_{\Gamma} G(s,x-y) \varphi(y), \quad x\in \Gamma.
\end{align*}
The solution of the above scattering problem \eqref{eq:Helmholtz} is then determined by
\begin{align*}
	\wu = -S(s)V(s)^{-1} \wg\inc.
\end{align*}
The boundary operators fulfill $s$-explicit bounds, which are described in the following Lemma.
\begin{lemma}\cite[Proposition~3]{B86I}
 The single-layer boundary operator is bounded by
\begin{align*}
	\left\| V(s) \right\|_{ H^{1/2}(\Gamma)\leftarrow \widetilde{H}^{-1/2}(\Gamma)}&\le C_\Gamma \dfrac{ |s|^2 + 1}{\mathrm{Re}\, s},
    \\
    	\left\| V(s)^{-1} \right\|_{\widetilde{H}^{-1/2}(\Gamma)\leftarrow H^{1/2}(\Gamma)}&\le C_\Gamma\dfrac{ |s|^2 + 1}{\Re s},
\end{align*}
where $C$ only depends on the geometry.
\end{lemma}
An explicit estimate on the constant $C_\Gamma$ in terms of the norms of trace operators with a proof is found in \cite[Lemmas 2.7 \& 2.8]{N23}.

\subsection{Boundary element discretization} The weak formulation of the boundary integral equation determining $\widehat\varphi$ with test functions in the appropriate test space $\widetilde{H}^{-1/2}(\Gamma)$ yields the following weak formulation: Find  $\widehat \varphi \in  \widetilde{H}^{-1/2}(\Gamma)$, such that for all $\widehat{v} \in  \widetilde{H}^{-1/2}(\Gamma)$, 
\begin{align*}
	\left\langle\widehat{v}, V(s) \widehat{\varphi}\right\rangle_\Gamma 
	& = 
	-\left\langle\widehat{v}, \widehat u\inc  \right\rangle_\Gamma.
\end{align*}
Consider the space of continuous piecewise constant elements $X_h\subset \widetilde{H}^{-1/2}(\Gamma)$. On this finite dimensional subspace, we formulate the Galerkin approximation: find $\widehat \varphi_h \in X_h \subset \widetilde{H}^{-1/2}(\Gamma)$, such that for all $\widehat{v}_h \in X_h \subset \widetilde{H}^{-1/2}(\Gamma)$
\begin{align*}
	\left\langle\widehat{v}_h, V(s) \widehat{\varphi}_h \right\rangle_\Gamma 
	& = -
	\left\langle\widehat{v}_h, \widehat u\inc  \right\rangle_\Gamma.
\end{align*}
Here, $ \left\langle\cdot,\cdot \right\rangle_\Gamma$ denotes the extension of the $L^2$-pairing onto the $\widetilde{H}^{-1/2}(\Gamma)\times H^{1/2}(\Gamma)$.
Due to the positive real part of $s$, the bilinear form of the right-hand side is coercive, which enables a classical a priori error analysis. Using the Lax-Milgram Lemma on the subspace $X_h$ shows that the spatially discrete resolvent $V_h(s)^{-1}$, defined by
\begin{align*}
\widehat\varphi_h = -V_h(s)^{-1}\widehat u\inc .
\end{align*}
inherits the Laplace domain bounds of the spatially continuous single-layer operator, in particular
\begin{align}\label{eq:V_h-inv}
\left\| V_h(s)^{-1} \right\|_{\widetilde{H}^{-1/2}(\Gamma)\leftarrow H^{1/2}(\Gamma)}&\le C_\Gamma\dfrac{ |s|^2 + 1}{\Re s}   .
\end{align}

\subsection{A posteriori error estimate}
We then find the following time-harmonic a posteriori estimate.
\begin{align}
\begin{aligned}\label{eq:time-harmonic-a-posteriori}
\left\|\widehat \varphi - \widehat \varphi_h \right\|_{\widetilde{H}^{-1/2}(\Gamma)}
& =
\left\|V(s)^{-1}\left(V(s) \widehat \varphi - V(s)\widehat \varphi_h \right)\right\|_{\widetilde{H}^{-1/2}(\Gamma)}
\\ & \le 
C_\Gamma \dfrac{ |s|^2 + 1}{\Re s}
\left\|\widehat f- V(s)\widehat \varphi_h \right\|_{H^{1/2}(\Gamma)}.
\end{aligned}
\end{align}
Let $\widehat{R}_h= \widehat f- V(s)\widehat \varphi_h\in H^{1/2}(\Gamma) $ denote the time-harmonic residual. We note that by following the argument structure above, we further obtain 
\begin{align}
\begin{aligned}\label{eq:time-harmonic-a-posteriori2}
\left\|R_h(s)\right\|_{H^{1/2}(\Gamma)}
& =
\left\|V(s)\left(V(s)^{-1} \widehat f -\widehat \varphi_h \right)\right\|_{H^{1/2}(\Gamma)}
\\ & \le C_\Gamma \dfrac{ |s|^2 + 1}{\Re s}
\left\|\widehat \varphi - \widehat \varphi_h \right\|_{\widetilde{H}^{-1/2}(\Gamma)}.
\end{aligned}
\end{align}
Overall, we therefore arrive at the following equivalence of the norms of the error and of the residual, which holds up to a factor depending on the Laplace parameter $s$
\begin{align}
\dfrac{\Re s}{C_\Gamma (|s|^2 + 1)}\left\|R_h(s)\right\|_{H^{1/2}(\Gamma)} 
\le 
\left\|\widehat \varphi - \widehat \varphi_h \right\|_{\widetilde{H}^{-1/2}(\Gamma)}
\le 
C_\Gamma \dfrac{ |s|^2 + 1}{\Re s}\left\|R_h(s)\right\|_{H^{1/2}(\Gamma)}.
\end{align}

\section{The time-dependent scattering problem}
\subsection{Temporal convolution operators}\label{sect:temporal_convolution}
We describe the techniques to transfer the Laplace domain results to the time-domain. 
Let $V$ be an Hilbert space and consider an analytic family of bounded linear operators $K(s)\colon V \rightarrow W$, which are defined for $\Re s>0$. We assume that $K$ is \textit{polynomially bounded} with respect to the Laplace domain parameter $s$, in the following sense: There exist a real $\kappa$, $\nu\ge 0$ and, for every $\sigma >0$, there exists $M_\sigma <\infty$, such that
\begin{align}\label{eq:pol_bound}
\norm{K(s)}_{W\leftarrow V}&\leq M_\sigma \dfrac{\abs{s}^\kappa}{(\Re s)\nu}, \quad\ \text{ Re } s \ge \sigma>0.
\end{align}
For a sufficiently regular function $g:[0,T]\to X$, which vanishes together with its first $m>\kappa$ derivatives, we write
\begin{equation} \label{Heaviside}
K(\partial_t) g = (\mathcal{L}^{-1}K) * g
\end{equation}
for the temporal convolution of the inverse Laplace transform of $K$ with~$g$, where $g$ is extended by zero on the negative real axis. 
Due to the associativity of convolution we obtain, for two such families of operators $K(s)$ and $L(s)$ mapping into compatible spaces, the composition rule
\begin{equation}\label{comp-rule}
K(\partial_t)L(\partial_t)g = (K L)(\partial_t)g.
\end{equation}
We denote the Sobolev space of real   order $r$ of $V$-valued functions on $\mathbb R$ by $H^r(\mathbb R,V)$ and, on finite intervals $(0, T )$,
we write
$$
H_0^r(0,T; V) = \{ g|_{(0,T)} \,:\, g \in H^r(\mathbb R, V)\ \text{ with }\ g = 0 \ \text{ on }\ (-\infty,0)\} .
$$
The natural norm on $H_0^r(0,T;V)$ is, for $r\ge 0$, equivalent to the norm $\| \partial_t^r g \|_{L^2(0,T;V)}$.
By the Plancherel formula, the temporal convolution operators $K(\partial_t)$ are bounded on the appropriate temporal Sobolev spaces. \cite[Lemma 2.1]{L94}:
If $K(s)$ is bounded by \eqref{eq:pol_bound} in the half-plane $\text{Re }s > 0$, then $K(\partial_t)$ fulfills
\begin{equation}\label{sobolev-bound}
\| K(\partial_t) \|_{ \H^{r}_0(0,T;V') \leftarrow \H^{r+\kappa}_0(0,T;V)} \le e M_{1/ T} T^\nu
\end{equation}
for arbitrary real $r$. (The bound on the right-hand side arises from the bound $e^{\sigma T} M_\sigma$ on choosing $\sigma=1/T$.)
We note that for any integer $k\ge 0$ and real ${\alpha>\tfrac12}$, we have the continuous embedding $\H^{k+\alpha}_0(0,T;V)\subset C^k([0,T];V)$.
 
\subsection{Time-dependent boundary integral equations}
We turn towards the time-dependent scattering problem, whose solution is determined by the boundary integral equation
\begin{align*}
V(\partial_t)\varphi = - u\inc .
\end{align*}
The scattered wave can be evaluated,  in the exterior domain $\mathbb R^3\setminus \Obs$, by the representation formula
\begin{align*}
	u = S(\partial_t)\varphi .
\end{align*}
By \eqref{sobolev-bound} and Laplace domain bounds, we obtain bounds on the (temporal) regularity of the scattered wave $u$, under assumptions on the incoming wave $u\inc$.
\subsection{Boundary element discretization}
The spatially discrete time-dependent boundary density is constructed by employing the Galerkin approach on the temporal convolution equation, which yields the following formulation:\\
Find the spatially discrete boundary density $ \varphi_h(t) \in  X_h$ (for $t\in [0,T]$), such that 
\begin{align*}
	\left\langle v_h, V(\partial_t)\varphi_h \right\rangle_\Gamma 
	& = 
	-\left\langle v_h, \widehat u\inc  \right\rangle_\Gamma
    \quad \text{for all}\quad  v_h \in  X_h.
\end{align*}
An equivalent and compact characterization of this approximation can be formulated with the temporal convolution operator associated to $V_h(s)^{-1}$, namely we have
\begin{align*}
\varphi_h = - V_h(\partial_t)^{-1}u\inc .
\end{align*}
We then have the following a posteriori error estimate of the spatially discrete system.
\begin{proposition}\label{prop:semi-discrete-space}
Let $\omega>0$ be an arbitrary constant. Then, the error of the spatial-semidiscretization is bounded, for all $r\in \mathbb N$, by
\begin{align*}
\int_0^\infty \left \|\varphi-\varphi_h \right\|_{L^2(0,T; \widetilde{H}^{-1/2}(\Gamma))}
&\le
C \big(1+\omega^2\big)\left\|f- V(\partial_t) \varphi_h \right\|_{L^2(0,T;H^{1/2}(\Gamma))}
\\&
+\dfrac{C}{\omega^r} 
\left\|f \right\|_{H^{r+6}_0(0,T;H^{1/2}(\Gamma))},
\end{align*}
if the right-hand side is finite.
\end{proposition}
\begin{proof}
 We start the proof by showing the stated estimate on a weighted $L^2$- norm on the positive half line.

\emph{(i) Proving the estimate in the $L^2(\mathbb{R}_+, e^{-2\sigma t}dt)$-norm.}
In the following, we choose $\sigma=1/T$ and use the Plancherel formula several times. We start from the standard estimate
\begin{align*}
\int_{0}^\infty e^{-2\sigma t} \left \|\varphi-\varphi_h \right\|^2_{\widetilde{H}^{-1/2}(\Gamma)} \, \mathrm d t
\le
C_\Gamma T\bigg(\int_{0}^\infty &e^{-2\sigma t}
\left\|\partial_t^2f- V(\partial_t) \partial_t^2\varphi_h \right\|^2_{H^{1/2}(\Gamma)}
\\  + &
e^{-2\sigma t}\left\|f- V(\partial_t) \varphi_h \right\|^2_{H^{1/2}(\Gamma)}\, \mathrm d t \bigg) .
\end{align*}
We note that the residual is measured in the temporal $H^2$ norm, which we intend to avoid, at the cost of introducing a data oscillation term. Following \cite{cf23}, we split the corresponding contour integral in the frequency domain and obtain, for any $\omega>0$, the estimate (for $\xi_h = f- V(\partial_t)\varphi_h $)
\begin{align*}
&\int_0^\infty e^{-2\sigma t} \left \|\partial_t^2 \xi_h \right\|_{ H^{1/2}(\Gamma)}^2
\le
\int_{\sigma+i\mathbb R} |s|^4\left \|\widehat{\xi}_h \right\|^2_{ H^{1/2}(\Gamma)}\, \mathrm d s
\\
&
\le \omega^4
\int_{\sigma+i\mathbb R} \left \|\widehat{\xi}_h \right\|^2_{ H^{1/2}(\Gamma)} \chi_{|s|<\omega}\, \mathrm d s
+\dfrac{1}{\omega^{2r}}\int_{\sigma+i\mathbb R} |s|^{4+2r}\left \|\widehat{\xi}_h \right\|^2_{ H^{1/2}(\Gamma)} \chi_{|s|>\omega}\, \mathrm d s
\\
&
\le \omega^4\int_0^\infty e^{-2\sigma t} 
\left \| \xi_h \right\|_{ H^{1/2}(\Gamma)}^2 \,\mathrm d t
+
\dfrac{1}{\omega^{2r}}
\int_0^\infty e^{-2\sigma t} 
\left \| \partial_t^{r+2} \xi_h \right\|_{ H^{1/2}(\Gamma)}^2\,\mathrm d t.
\end{align*}
Under sufficient temporal regularity of the right-hand side, we can therefore split the error estimate into
\begin{align*}
\int_0^\infty e^{-2\sigma t} \left \|\varphi-\varphi_h \right\|^2_{ \widetilde{H}^{-1/2}(\Gamma)}\,\mathrm d t
&\le
C \big(1+\omega^2\big)\int_0^\infty e^{-2\sigma t} \left\|f- V(\partial_t) \varphi_h \right\|^2_{H^{1/2}(\Gamma)}\,\mathrm d t
\\&
+\dfrac{C}{\omega^r} 
\int_0^\infty e^{-2\sigma t} 
\left\|\partial_t^{r+2}f- V(\partial_t) \partial_t^{r+2}\varphi_h \right\|^2_{H^{1/2}(\Gamma)}
\,\mathrm d t.
\end{align*}
The final term can be estimated by
\begin{align*}
\int_0^\infty e^{-2\sigma t} \left\|\partial_t^{r+2}(I-V(\partial_t)V_h^{-1}(\partial_t))f \right\|^2_{H^{1/2}(\Gamma)} \,\mathrm dt
\le
C\int_0^\infty e^{-2\sigma t} \left\|\partial^{r+6}_tf \right\|^2_{H^{1/2}(\Gamma)} \, \mathrm dt,
\end{align*}
which gives the stated result in the weighted norm on $\mathbb R_+$.

\emph{(ii) Transferring the result to $L^2(0,T)$.}
Stein's extension theorem \cite[Theorem 5.24]{AF03}, shows the existence of an extension operator that is bounded by
$$ \|E \|_{H_0^{k}(\mathbb R_+;H^{1/2}(\Gamma))\leftarrow H_0^{k}(0,T;H^{1/2}(\Gamma)) } \le C,$$
for arbitrary $k$. 
We extend both $\xi_h$ and $f$ with this operator and further extend $\varphi_h$ and $\varphi$ by defining 
\begin{align*}
\widetilde{\varphi} &=  V(\partial_t)^{-1}(E(f) )\quad \text{and}
,\quad
\widetilde{\varphi}_h =  V(\partial_t)^{-1}(E(f)- E(\xi_h) ).
\end{align*}
By causality, these functions coincide with $\varphi$ and $\varphi_h$ on $(0,T)$.
Applying the bound of part (i) and the continuity of the extension operator now yields the stated result by estimating
\begin{align*}
&\left \|\varphi-\varphi_h \right\|_{L^2(0,T; \widetilde{H}^{-1/2}(\Gamma))}^2\le e
 \int_{0}^\infty e^{-2\sigma t}\left\|\widetilde{\varphi}-\widetilde{\varphi_h} \right\|^2_{ \widetilde{H}^{-1/2}(\Gamma)}\, \mathrm d t
 \\ & 
\le C \int_0^\infty e^{-2\sigma t}\big(1+\omega^2\big)\left\|E(\xi_h) \right\|^2_{H^{1/2}(\Gamma)}
+\dfrac{e^{-2\sigma t}}{\omega^{2r}} 
\left\|\partial_t^{r+6} E(f) \right\|^2_{H^{1/2}(\Gamma)}
\, \mathrm d t
\\ & \le 
C \big(1+\omega^2\big)\left\|f- V(\partial_t) \varphi_h \right\|_{L^2(0,T;H^{1/2}(\Gamma))}^2
+\dfrac{C}{\omega^r} 
\left\|f \right\|_{H^{r+6}_0(0,T;H^{1/2}(\Gamma))}^2.
\end{align*}
\end{proof}

\section{Full discretization}
In the following, we extend our analysis to temporally discrete systems.
\subsection{Time discretization by Convolution Quadrature}

We recall the approximation of temporal convolutions $K(\partial_t)g$ by the convolution quadrature method based on underlying time stepping scheme and fix the associated notation. 
With the constant time step size $\tau>0$, we construct approximations $K(\partial_t^\tau)g^n$ to $K(\partial_t)g(t_n)$ at equidistant time points $t_n = n\tau$.  
Most effective convolution quadrature schemes are based on Runge--Kutta methods, which are fully determined by their Butcher-tableau
\begin{equation*}
	\mathscr{A} = (a_{ij})_{i,j = 1}^m \,, \quad b = (b_1,\dotsc,b_m)^T\, ,
	\quad \text{and} \quad c = (c_1,\dotsc,c_m)^T \,.
\end{equation*}
The associated stability function of the Runge--Kutta method is $R(z) = 1 + z b^T ( I - z \mathscr{A})^{-1} \mathbbm{1}$, where $\mathbbm{1} = (1,1,\dotsc,1)^T \in \mathbb{R}^m$. 
Consider an temporal convolution operator in the sense of Section~\ref{sect:temporal_convolution}, namely $\cqK(\partial_t):H^{r+\kappa}_0(0,T;V) \to \cqH^{r}_0(0,T;W)$ for arbitrary real $r$. Moreover, let $\cqg:[0,T]\to \cqX$  be a time-dependent function that is sufficiently regular and vanishes, together with sufficiently many derivatives at $t=0$ for the expression \eqref{Heaviside} to be well-defined. 
The convolution $(\cqK(\partial_t)\cqg)(t)$ is approximated at the stages of the Runge--Kutta methods, namely at the discrete times
$$
\vtn = (t_n+c_\ell \tau)_{\ell=1}^m\, ,\text{ where } t_n=n\tau\, .
$$
The Runge--Kutta differentiation symbol is defined by the expression
\begin{equation*}
	\Delta(\zeta) \, := \, \Bigl(\mathscr{A}+\frac\zeta{1-\zeta}\mathbbm{1} \cqb^T\Bigr)^{-1} \in \mathbb{C}^{m \times m}\, , \qquad
	\zeta\in\mathbb{C} \hbox{ with } |\zeta|<1\, .
\end{equation*}
As a consequence of the Sherman--Woodbury formula, this expression is well defined for $|\zeta|<1$ if $R(\infty)=1-b^T\mathscr{A}^{-1}\bone$ satisfies $|R(\infty)|\le 1$. For A-stable Runge--Kutta methods (e.g.\@ the Radau IIA methods), the eigenvalues of the matrices $\Delta(\zeta)$ have positive real part for $|\zeta|<1$ (see \cite[Lem.\@ 3]{BLM11}).
The Sherman--Morrison formula then yields the expression
\begin{equation*}
	\Delta(\zeta) \,=\, \mathscr{A}^{-1} -\frac{\zeta}{1-R(\infty)\zeta} \mathscr{A}^{-1} \bone b^T \mathscr{A}^{-1}\, .
\end{equation*}
We are now in a position to define the convolution quadrature weights ${\cqW}_n(\cqK):\cqX^m \to \cqY^m$.

We replace the complex argument $s$ in $\cqK(s)$ by the matrix-valued analytic function $\Delta(\zeta)/\tau$ and write down the power series expansion
\begin{equation*}
	\cqK\Bigl(\frac{\Delta(\zeta)}\tau \Bigr) \,=\, \sum_{n=0}^\infty {\boldsymb W}_n(\cqK) \zeta^n\,.
\end{equation*}
 In the following, we use an upper index to denote a sequence element with $m$ components.
Thus, for a sequence $\cqg=(\cqg^n)$ with $\cqg^n=(\cqg^n_\ell)_{\ell=1}^m\in \cqX^m$ we arrive at the discrete convolution denoted by
\begin{equation}\label{rkcq}
	\bigl(\cqK(\underline{\partial_t^\tau}) \boldsymb g \bigr)^n \, := \, \sum_{j=0}^n {\boldsymb W}_{n-j}(\cqK) \boldsymb g^j \in \cqY^m \,.
\end{equation}
The notation $\cqK(\underline{\partial_t^\tau}) \boldsymb g$ in \eqref{rkcq} indicates that the resulting vector contains approximations at the stages $\vtn$.
For functions $\boldsymb g:[0,T]\to \cqX$, we use this notation for the vectors $\cqg^n = \cqg(\vtn) = \bigl(\cqg(t_n+c_i\tau)\bigr)_{i=1}^m$ of values of $\cqg$.
The $\ell$-th component of the vector $\bigl(\cqK(\underline{\partial_t^\tau}) \boldsymb g \bigr)^n$, that we denote by $\bigl(\cqK(\underline{\partial_t^\tau}) \boldsymb g \bigr)^{n,\ell}$, is then an approximation to $\bigl(\cqK(\partial_t)\cqg\bigr)(t_n+c_\ell\tau)$, i.e. $\bigl(\cqK( \underline{\partial_t^\tau}) \boldsymb g \bigr)^{n,\ell} \approx \bigl(\cqK(\partial_t) \cqg \bigr)(t_{n}+c_\ell \tau)$
 (see \cite[Thm.\@ 4.2]{BL19}). 
In particular, if $c_m = 1$, as is the case with stiffly stable Runge--Kutta methods, which includes the Radau IIA methods,
the continuous convolution at $t_{n}$ is approximated by the $m$-th, i.e. last component of the $m$-vector \eqref{rkcq} for $n-1$:
\begin{equation*}
	 \left(\cqK( \partial_t^\tau) \boldsymb g\right)_n \, := \,  \bigl(\cqK( \underline{\partial_t^\tau}) \boldsymb g \bigr)^{n-1,m} \in \cqY \, . 
\end{equation*}
These components approximate the continuous convolution at the equidistant time points $t_n$, i.e. $ \left(\cqK( \partial_t^\tau) \boldsymb g\right)_n\approx\bigl(\cqK(\partial_t) \cqg \bigr)(t_{n}) $.

A property that is key to show stability in many settings is that the composition rule~\eqref{comp-rule} is preserved under this discretization: For two compatible operator families $\cqK(s)$ and $\cqL(s)$, we have
\begin{equation*}
	\cqK(\underline{\partial_t^\tau})\cqL(\underline{\partial_t^\tau})\cqg \, =\,  (\cqK\cqL)(\underline{\partial_t^\tau})\cqg \, .
\end{equation*}
Such a property can only be formulated for the vector valued discrete convolution, which includes the approximations at the stages $\vtn$, but can not be formulated for the approximation $\cqK( \partial_t^\tau) \boldsymb g $ at the equidistant time points $t_n$.
\begin{remark}[Multistep based convolution quadrature]\label{rem:multistep-cq}
The setting of multistep-based convolution quadrature is the technically simplest setting. The convolution quadrature weights are then determined by the power series 
\begin{align*}
K\left(\dfrac{\delta(\zeta)}{\tau}\right)
&=
\sum_{n=0}^\infty W_n(K)\zeta^n,
\end{align*}
where $\delta$ is the scalar-valued characteristic function of the multistep method, which is in the case of the BDF-$p$ method given by 
\begin{align*}
\delta(\zeta) = \sum_{\ell=1}^p \dfrac{1}{\ell}(1-\zeta)^\ell \quad,\text{for} \quad p=1,2 .
\end{align*}
For the method to be A-stable, we require $p\le 2$.
\end{remark}

The following error bound for Runge--Kutta convolution quadrature from~\cite[Thm.\@ 3]{BLM11}, here directly stated for the Radau IIA methods \cite[Sec.~IV.5]{hairerwannerII} and transferred to a Hilbert space setting, will be the basis for our error bounds of the  time discretization.

\begin{lemma}
	\label{lem:RK-CQ}
	Let $\cqK(s):\cqX\to \cqY$, $\real s \geq \sigma>0$,  be an analytic family of linear operators between Banach spaces $\cqX$ and $\cqY$ satisfying
	the bound \eqref{eq:pol_bound} with exponents $\kappa$ and $\nu$.
	Consider the Runge--Kutta convolution quadrature based on the Radau IIA method with $m$ stages. Let $r>\max(2m+\kappa,2m-1,m+1)$ and further let $\cqg \in \cqC^r([0,T],\cqX)$ satisfy $\cqg(0)=\cqg'(0)=...=\cqg^{(r-1)}(0)=0$. Then, the following error bound holds at $t_n=n\tau\in[0,T]$:
	\begin{multline*}
	\left\|  (\bigl(\cqK( \partial_t^\tau) \boldsymb g \bigr)_i^{n})_{i=1,\dots,m}-(\cqK(\partial_t)\cqg)(\vtn ) \right\|_{\cqY^m}
		\\ \, \le \, 
		C\, M_{1/T}\,\tau^{\min(2m-1,m+1-\kappa+\nu)}
		\left(\|{\cqg^{(r)}(0)}\|_{\cqX}+\int_0^t\|{\cqg^{(r+1)}(t')}\|_{\cqX} \,\mathrm{d}t'
		\right) \, .
	\end{multline*}
	The constant C is independent of $\tau$ and $\cqg$ and $M_\sigma$ of \eqref{eq:pol_bound}, but depends on the exponents $\kappa$ and $\nu$ in \eqref{eq:pol_bound} and on the final time $T$.
\end{lemma}
\subsection{Modified formulation}\label{sec:modified}

In the context of transfer operators that are bounded only by a high power of $|s|$ (i.e. that fulfill \eqref{eq:pol_bound} with a moderately large $\kappa$), the Runge--Kutta convolution quadrature approximation suffers from order reductions. We use the following modification to the convolution quadrature method to overcome this difficulty. Let $g\in H_0^1(\mathbb R_+;V)$ be some arbitrary density, for which its Laplace transform
$$\mathcal L (g(\cdot-\eta)) (s)= \int_0^\infty e^{-st}g(t-\eta) \, \mathrm d t
=
e^{-s\eta}\int_{-\eta}^\infty e^{-st}g(t) \, \mathrm d t
=e^{-s\eta}\mathcal L(g) (s)
.$$
We observe that the $\eta-$shift in time-domain is a multiplication with $J_\eta(s)=e^{-s\eta}$ in the Laplace domain.
The boundary integral equation can therefore be modified, 
\begin{align*}
    \varphi 
    &=
    \left( J_\eta (\partial_t)V(\partial_t)^{-1} \right)    J_{-\eta} (\partial_t)f.
\end{align*}
The modification 
\begin{align}
    J_{-\eta} (\partial_t)f(t) = 
\begin{cases}
    0 ,&\text{for} \quad t\le \sigma ,
    \\    f(t+\eta)  , &\text{for} \quad t>\sigma .
\end{cases}
\end{align}
can be realized exactly, by appending the function by zeros (and is valid since $f$ is assumed to vanish on the negative real-axis). In the Laplace domain, we now have that 
\begin{align}\label{eq:shifted-s-est}
 \| J_\eta (s)V(s)^{-1}\|_{\widetilde{H}^{-1/2}(\Gamma)\leftarrow H^{1/2}(\Gamma)} 
 \le C_\Gamma \dfrac{|s|^2+1}{\mathrm{Re} \,s} e^{-\eta \mathrm{Re}\, s} .
\end{align}
The shifted scheme then reads
\begin{align}\label{eq:time-discr-shift}
\varphi^\tau_\eta
&=
\left( V(\partial^\tau_t)^{-1} \right)   J_\eta (\partial_t^\tau) J_{-\eta} (\partial_t)\underline{f}.
\end{align}
Note that in order to construct a useful solution to the original convolution equation, we have to compute the solution to the convolution equation \eqref{eq:time-discr-shift} until the final time $T+\eta$ and then disregard the numerical solution on the extended interval $[T,T+\eta]$. 
This modification has two desirable effects:

\emph{Full classical convergence rate.} The exponential decay in the Laplace domain estimate \eqref{eq:shifted-s-est} improves corresponding the error bound Lemma~\ref{lem:RK-CQ} and shows that the shifted density $\varphi^\tau_\eta$ converges at the full classical order $\tau^{2m-1}$. Without any shift, i.e. $\eta=0$, this reduces to the stage order $\tau^m$.

\emph{Fast evaluation due to the truncation of damped wave numbers.} The ``all-at-once" implementation of the convolution quadrature method reduces the computational bottleneck to the evaluation of the Laplace domain operator $J_\eta (s)V(s)^{-1}$ at several different frequencies $s_j$ (see e.g. \cite[Section 3.2]{BS09}). Due to the exponential decay in the bound \eqref{eq:shifted-s-est}, the operator can be truncated for Laplace domain arguments $s$ with real part that is larger than some predetermined constant.

\subsection{Fully discrete a posteriori error analysis}
For arbitrary real-valued order $r>0$, we have the following chain of equations 
\begin{align}
\begin{aligned}\label{eq:time-harmonic-a-posteriori3}
\left\| \varphi^\tau -\varphi^\tau_h \right\|_{H^r_0(0,T; \widetilde{H}^{-1/2}(\Gamma))}
& =
\left\|V(\underline{\partial_t^\tau})^{-1}\left(V(\underline{\partial_t^\tau})  \varphi^\tau - V(\underline{\partial_t^\tau}) \varphi^\tau_h \right)\right\|_{H^r_0(0,T; \widetilde{H}^{-1/2}(\Gamma))}
\\ & \le C 
\left\|(\underline{\partial_t^\tau})^{2}\left(J_{-\eta}(\underline{\partial_t^\tau})J_{\eta}(\partial_t)\underline{f} - V(\underline{\partial_t^\tau}) \varphi^\tau_h \right)\right\|_{H^r_0(0,T; \widetilde{H}^{-1/2}(\Gamma))}.
\end{aligned}
\end{align}
Here, we used the operator valued bound
\begin{align*}
\left\|V(\underline{\partial_t^\tau})^{-1}
(\underline{\partial_t^\tau})^{-2}\right\|_{H^r_0(0,T;\widetilde{H}^{-1/2}(\Gamma))\leftarrow H^r_0(0,T;H^{1/2}(\Gamma))} \le C .
\end{align*}
By combining the estimate with the temporal semi-discrete error bound yields the following error estimate.
\begin{theorem}\label{th:full}
Under temporal regularity on the exact solution, we have the estimate
\begin{align*}
\left\| \varphi -\varphi^\tau_h \right\|_{H^r_0(0,T; \widetilde{H}^{-1/2}(\Gamma))}
&  \le C 
\left\|(\underline{\partial_t^\tau})^{2}\left(J_{-\eta}(\underline{\partial_t^\tau})J_{\eta}(\partial_t)\underline{f} - V(\underline{\partial_t^\tau}) \underline{\varphi^\tau_h} \right)\right\|_{H^r_0(0,T; \widetilde{H}^{-1/2}(\Gamma))}
\\&+C\tau^{2m-1},
\end{align*}
where the constant $C$ depends on the geometry of $\Omega$, polynomially on the final time $T$ and on regularity of $f$.
\end{theorem}
For the stages, we further obtain
\begin{align*}
\left\| \underline{\varphi} -\underline{\varphi^\tau_h} \right\|_{H^r_0(0,T; \widetilde{H}^{-1/2}(\Gamma))}
&  \le C 
\left\|(\underline{\partial_t^\tau})^{2}\left(\underline{f} - V(\underline{\partial_t^\tau}) \varphi^\tau_h \right)\right\|_{H^r_0(0,T; \widetilde{H}^{-1/2}(\Gamma))}
+C\tau^{m+1}.
\end{align*}

\subsection{A posteriori error analysis of BDF methods}
 In the case of multistep methods, we can improve the a posteriori error estimate by using the stability of the convolution quadrature method.
\begin{theorem}\label{thm:full-no-derivatives}
Let $\omega>0$ be an arbitrary constant. Then, the space discretization error of the temporal semi-discretization (with the BDF-1 or the BDF-2 method) is bounded, for all $r\in \mathbb N$, by
\begin{align*}
\left \|\varphi^\tau-\varphi_h^\tau \right\|_{L^2(0,T; \widetilde{H}^{-1/2}(\Gamma))}
&\le
C \big(1+\omega^2\big)\left\|f- V(\partial_t^\tau) \varphi_h^\tau \right\|_{L^2(0,T;H^{1/2}(\Gamma))}
\\&
+\dfrac{C}{\omega^r} 
\left\|f \right\|_{H^{r+6}_0(0,T;H^{1/2}(\Gamma))},
\end{align*}
if the right-hand side is finite.
\end{theorem}
\begin{proof}
In the following, we write $R^\tau_h=f- V(\partial_t^\tau)\varphi^\tau_h $.
We (formally) start with 
\begin{alignat*}{2}
&\int_0^\infty e^{-2\sigma t}\left \|\varphi^\tau-\varphi_h^\tau \right\|^2_{\widetilde{H}^{-1/2}(\Gamma)} \, \mathrm d t
= 
\int_{\sigma+i\mathbb R}
\left\|\widehat{\varphi}^\tau-\widehat{\varphi}_h^\tau \right\|^2_{\widetilde{H}^{-1/2}(\Gamma)} \, \mathrm d s&&
\\&\le C
\int_{\sigma+i\mathbb R}
\left|\left(\dfrac{\delta(e^{-s\tau})}{\tau}\right)^2\right|^2
\left\|\widehat{f}-V\left(\dfrac{\delta(e^{-s\tau})}{\tau}\right)\widehat{\varphi}_h^\tau \right\|^2_{\widetilde{H}^{-1/2}(\Gamma)} \, \mathrm d s &&
\\&\le
C \sup_{s\in \sigma+i\mathbb R, |s|<\omega}\left|\left(\dfrac{\delta(e^{-s\tau})}{\tau}\right)^2\right|^2
\int_{\sigma+i\mathbb R}
\left\|\widehat{f}-V\left(\dfrac{\delta(e^{-s\tau})}{\tau}\right)\widehat{\varphi}_h^\tau \right\|^2_{\widetilde{H}^{-1/2}(\Gamma)} \chi_{|s|<\omega}\, \mathrm d s &&
\\&+
C \dfrac{1}{\omega^{2r}}
\int_{\sigma+i\mathbb R}
|s|^{2r}
\left|\left(\dfrac{\delta(e^{-s\tau})}{\tau}\right)^2\right|^2
\left\|\widehat{f}-V\left(\dfrac{\delta(e^{-s\tau})}{\tau}\right)\widehat{\varphi}_h^\tau \right\|^2_{\widetilde{H}^{-1/2}(\Gamma)} \chi_{|s|>\omega}\, \mathrm d s,
\\&\le
C \omega^4
\int_0^\infty e^{-2\sigma t} \left\| f- V(\partial_t^\tau)\varphi^\tau_h\right\|^2_{\widetilde{H}^{-1/2}(\Gamma)} \,\mathrm d t 
+
 \dfrac{C}{\omega^{2r}}
\int_{\sigma+i\mathbb R}
|s|^{2r+4+8}
\left\|\widehat{f} \right\|^2_{\widetilde{H}^{-1/2}(\Gamma)}\, \mathrm d s.
\end{alignat*}
The last estimate holds only in the case of $\delta(\zeta) = 1-\zeta$ (implicit Euler method), or the BDF-$2$ method, with \cite[Lemma 3.3]{BLN20}. The bound can then be carried over to the finite time interval $(0,T)$ by the extension argument of the proof of Proposition~\ref{prop:semi-discrete-space}.
\end{proof}

\subsection{Localized error estimator}

As in \cite{gs18} (pp.~340-342 for $2$d, pp.~349-352 for $3$d), the right hand side is estimated by a local, computable quantity:
\begin{align*}
		&\left\|\left(\partial^\tau_t\right)^kf- V(\partial_t^\tau) \left(\partial_t^\tau\right)^k\varphi^\tau_h \right\|^2_{l^2(0,T;H^{1/2}(\Gamma))}\\ &\leq \widetilde{C}_\Gamma \tau \sum_{t_i} \sum_{E_j} h_{E_j}\int_{E_j} \left(\left(\partial_t^\tau\right)^k\nabla_{\Gamma}(f(t_i,x)- V(\partial_t^\tau)\varphi^\tau_h(t_i,x))\right)^2 dx,
\end{align*}
where $k\in \mathbb{N}_0$, $h_{E_j}$ is the diameter of the element $E_j$ and $t_i = i \tau$. If relevant, the proof also extends to $H^{s}(\Gamma)$ for $s \in [0,1]$.\\

The space-adaptive algorithm relies on the error indicator $$\eta(E_j)^2 = \tau  h_{E_j}\sum_{k\in\{0,2\}}\sum_{t_i}\int_{E_j} \left(\left(\partial_t^\tau\right)^k\nabla_{\Gamma}(f(t_i,x)- V(\partial_t^\tau)\varphi^\tau_h(t_i,x))\right)^2 dx$$
in the element $E_j$ of the spatial mesh. Here, $\nabla_\Gamma$ denotes the surface gradient on $\Gamma$.

\section{Adaptive algorithm and implementation}\label{sec:adaptive}
We refer the reader to \cite{goss20} for more details on adaptive schemes, there presented in the context of a space-time Galerkin method. We follow the classical construction of adaptive finite element methods, which has the structure: \\

\begin{center}
		\begin{tikzpicture}[every node/.style={font=\ttfamily}, >=Stealth]
			
			\node[draw, rectangle] (solve) at (0,0) {solve};
			\node[draw, rectangle] (estimate) at (3,0) {estimate};
			\node[draw, rectangle] (mark) at (6,0) {mark};
			\node[draw, rectangle] (refine) at (9,0) {refine};
			
			\draw[->] (solve) -- (estimate);
			\draw[->] (estimate) -- (mark);
			\draw[->] (mark) -- (refine);
			
		\end{tikzpicture}
\end{center}
\
\\
More precisely, we iterate over the following steps, until the error indicator falls below a target tolerance:
\begin{enumerate}
	\item Solve the boundary integral equation \begin{align*}
		\left\langle v_h, V(\underline{\partial}_t^\tau) \varphi^\tau_h \right\rangle_\Gamma 
		& = 
		\left\langle v_h, J_\eta (\underline{\partial}_t^\tau) J_{-\eta} (\underline{\partial}_t) \underline{f}  \right\rangle_\Gamma.
	\end{align*}
	\item Compute error indicators $\eta(E_j)$,
        \item Use the Dörfler marking criterion \cite{D96} to mark those elements that have a relatively large error indicator, characterized that for some predetermined $\theta\in (0,1)$ it holds that
        \begin{align*}
            \eta(E_j) > \theta \eta_{\max},
        \end{align*}
        where $\eta_{\max}$ is the maximal error indicator of the mesh.
	\item Refine all marked elements $E_j$.
\end{enumerate}
The practical computation and implementation of the error estimator involves some numerical challenges, for which we follow \cite[Section~7.1]{bespalov2019adaptive}. \\

To be specific, for a given discrete solution $\varphi_h^{\tau}(t_i) \in X_h$ at time step $t_i = i\tau$  we project $V(\partial_t^{\tau})\varphi_h^{\tau}$ to a function $P_h(t_i) $ in the finite element space of piecewise linear functions $Y_h \subset H^1(\Gamma)$. We also determine the best approximation $f_h(t_i) \in Y_h$ for the right-hand side $f(t) \in H^{1/2}(\Gamma)$. Then $R_h^{\tau}(t_i) = P_h(t_i)  - f_h(t_i) \in Y_h$. For every element $E_j$ of the spatial mesh we approximate
\begin{equation}
    \eta(E_j)^2 \approx \eta_h(E_j)^2 \coloneqq \tau h_{E_j} \sum_{t_i} \int_{E_j} |\nabla_{\Gamma} R_h^{\tau}(t_i, x)|^2  dx,
\end{equation}
where $\eta_h(E_j)$ can be computed exactly by numerical quadrature. In the numerical experiments below, $5$ Gauss quadrature points are used per element.

\section{Numerical Experiments}

For the numerical experiments we consider a right hand side corresponding to an incoming plane wave. Define
\begin{equation*}
    f(t) = \sin(\omega (t - t_{\text{lag}}))H(t)H(L - t), \qquad \omega = 2,\ L = 2,
\end{equation*}
where 
\begin{equation*}
    H(t) = 1 - \dfrac{1}{1+\exp(\beta t)}, \qquad \beta = 5,
\end{equation*}
and $t_{\text{lag}} = 4$. For a given direction $d \in \mathbb{R}^2$ we define $u^{\text{inc}}$ as the plane wave
\begin{equation*}
    u^{\text{inc}}(x) \coloneqq f(x \cdot d - t), \qquad \text{ for all }x\in \mathbb{R}^2.
\end{equation*}
We solve the Dirichlet scattering problem up to the final time $T=10$ with a Galerkin boundary-element discretization in space with $\mathcal{P}_0$--basis functions and RKCQ with a two-step Radau IIA in time. The following schemes are compared: uniform refinement; adaptive refinement (described in Section \ref{sec:adaptive}); $\beta$--graded meshes with grading parameter $\beta \in \{2, 3\}$ \cite{gs18}. Point evaluations are computed using the representation formula at the points $x_1 = (2, 2), \ x_2 = (-2, 2), \ x_3 = (-2, -2)$ and $x_4 = (2, -2)$. The corresponding error is computed as the square root of the sum of the squares of errors in each of these points. \\

We measure the errors in the energy space $L^2(0, T; \widetilde{H}^{-1/2}(\Gamma))$ using the equivalent norm

\begin{equation}\label{eq:energy-norm} 
\begin{aligned}
    \norm{ \varphi_h^{\tau}}_{L^2(0, T; \widetilde{H}^{-1/2}(\Gamma))}^2 &= \int \norm{\varphi_h^{\tau}(t)}_{\widetilde{H}^{-1/2}(\Gamma)}^2 \ \mathrm{d}t = \int \left\langle\varphi_h^{\tau}(t), V(1)\varphi_h^{\tau}(t)\right\rangle_{\Gamma} \ \mathrm{d}t \\ 
    &\approx \tau \sum_{t_i} \left\langle\varphi_h^{\tau}(t_i), V(1)\varphi_h^{\tau}(t_i)\right\rangle_{\Gamma},
    \end{aligned}
\end{equation}
where $V(1)$ is the single-layer boundary integral operator with $s = 1.$

\subsection{Flat Screen}\label{sec:segment}
 We solve the problem at a flat screen $\Gamma \coloneqq (-1, 1)\times \{0\}$. We show a sequence of meshes obtained by adaptive refinement in Figure \ref{fig:mesh-refinement}. Results for convergence of point evaluation and convergence of densities in energy norm \eqref{eq:energy-norm} are presented in Figure \ref{fig:CQ1}. The error estimator underestimates the error roughly by an order of magnitude. Nevertheless, we observe that the adaptive refinement displays the same behavior as the asymptotically optimal $3$--graded mesh refinement. Figure~\ref{fig:CQ3} shows snapshots of the density $\varphi_h^{\tau}$ at times $t = 2, \ 4, \ 6$ and $8$. Near the endpoints, the solution behaves singularly during all the snapshot times, which is the key challenge that has to be overcome by the numerical approximation for an accurate solution. The sequence of refinements is visualized in Figure~\ref{fig:mesh-refinement}.
\begin{figure}[ht!]
    \centering
    \includegraphics[width=0.6\linewidth]{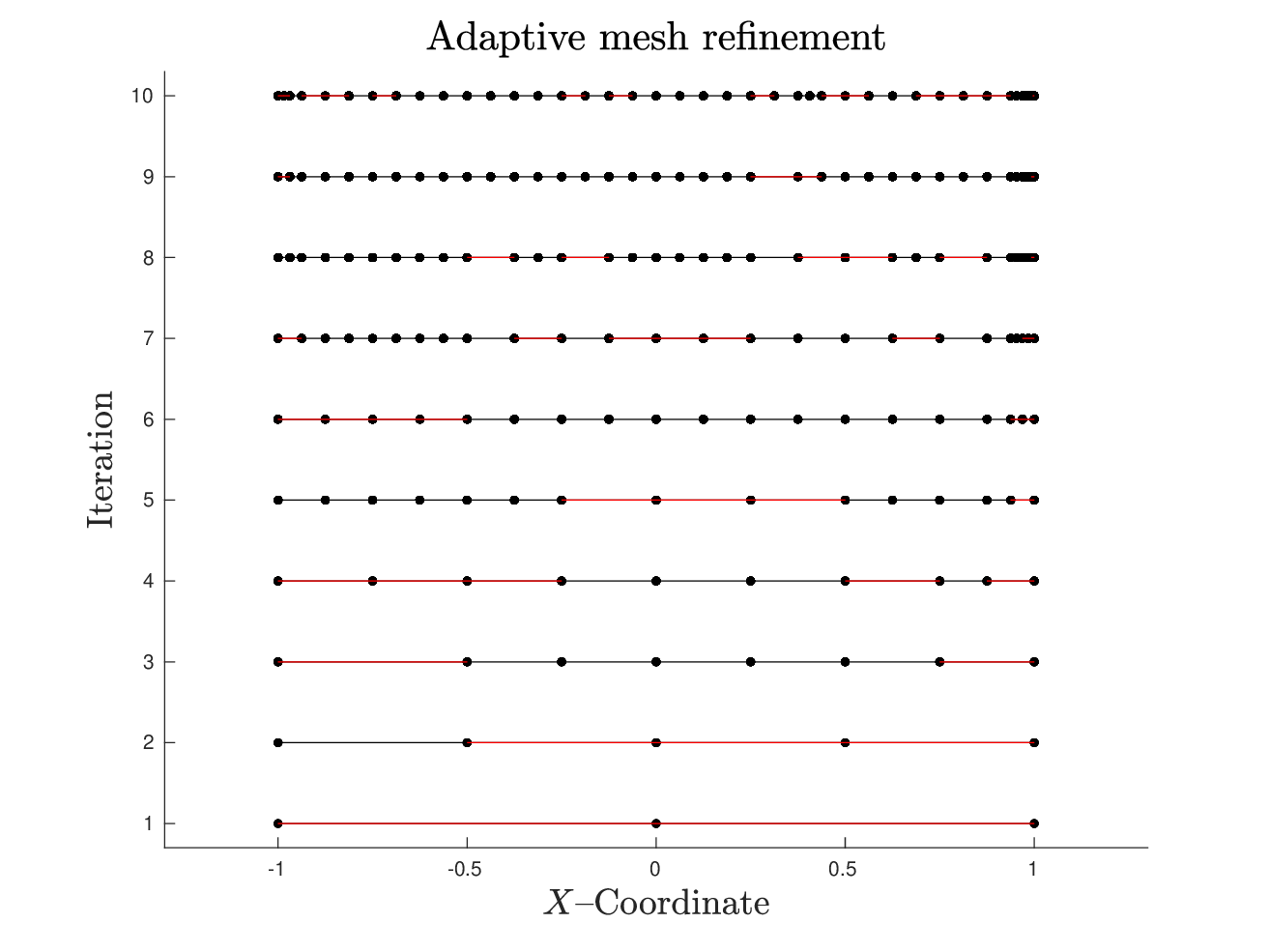}
    \caption{Scattering at a flat screen, $d = (-\tfrac{\sqrt{3}}{2}, \tfrac{1}{2})$: sequence of meshes obtained by adaptive refinement. The red lines denote the marked elements at each iteration.}
    \label{fig:mesh-refinement}
\end{figure}

\begin{figure}[ht!]
    \centering
    \subfloat[$\tau=0.1$]{
    \includegraphics[width=0.6\linewidth]{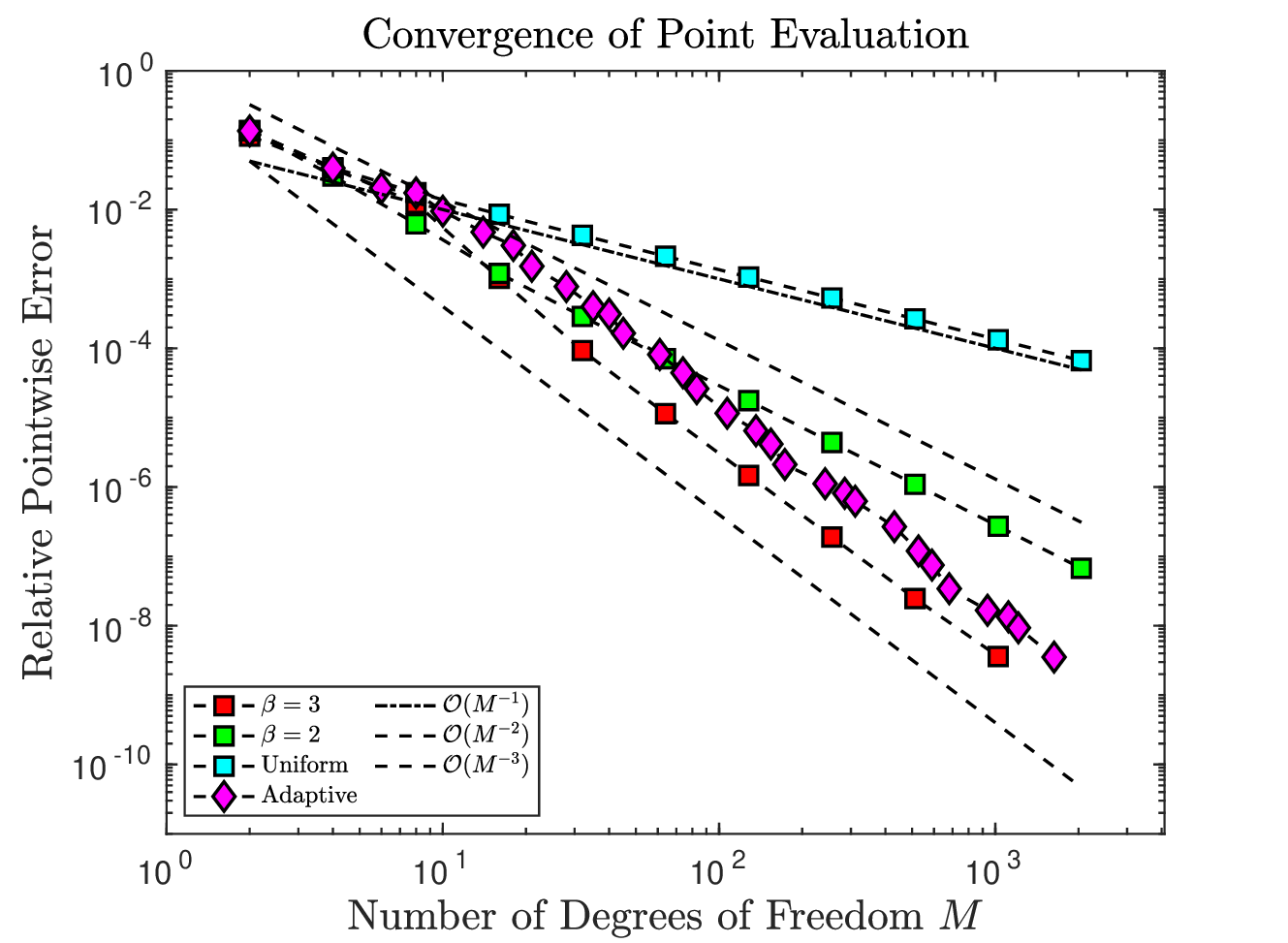}}
    \subfloat[$\tau=0.1$]{
    \includegraphics[width=0.6\linewidth]{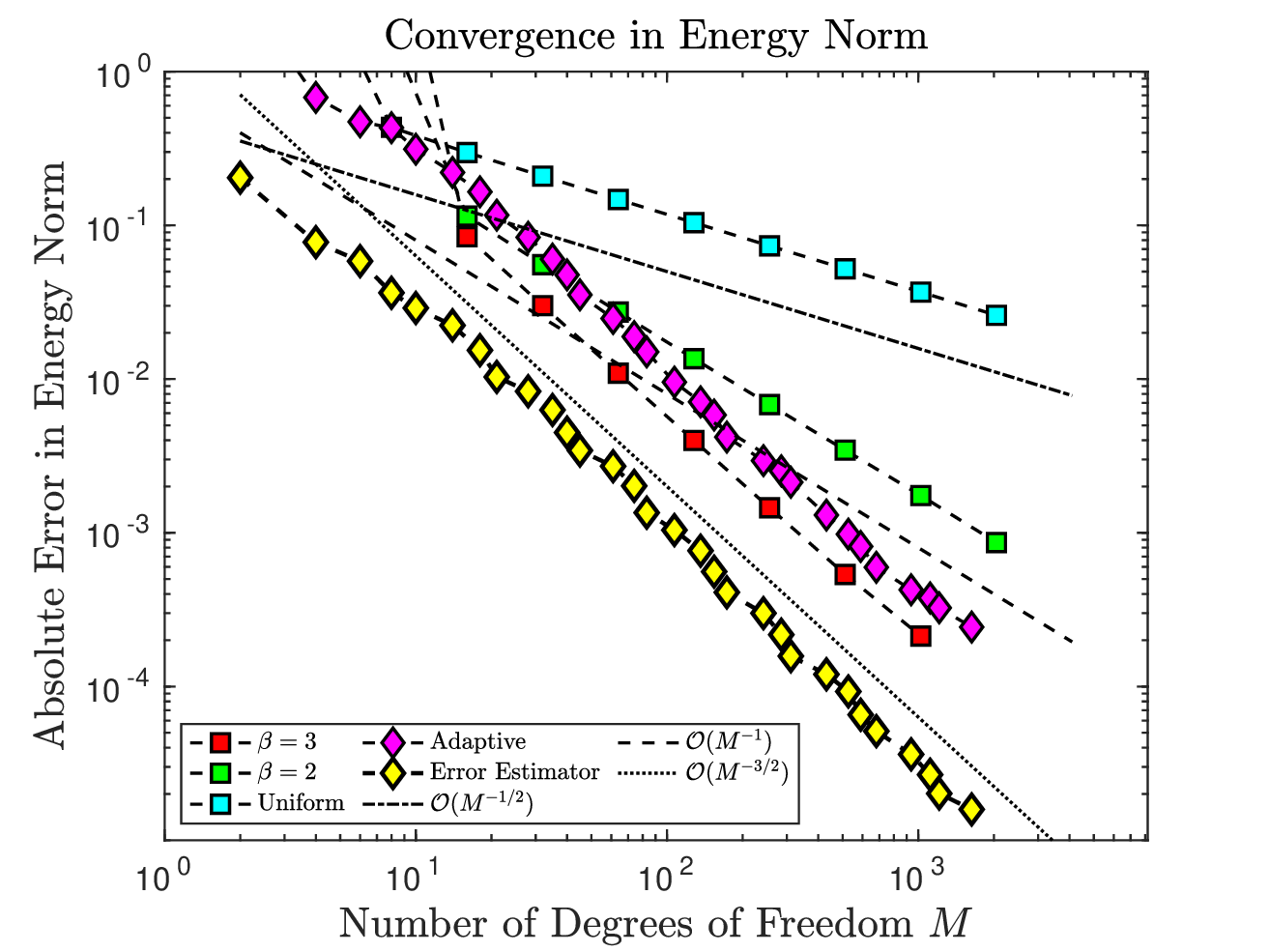}}
    \caption{Scattering at a flat screen, $d = (-\tfrac{\sqrt{3}}{2}, \tfrac{1}{2})$. }
    \label{fig:CQ1}
\end{figure}

\begin{figure}[ht!]
    \centering
    \subfloat[]{
    \includegraphics[width=0.5\linewidth]{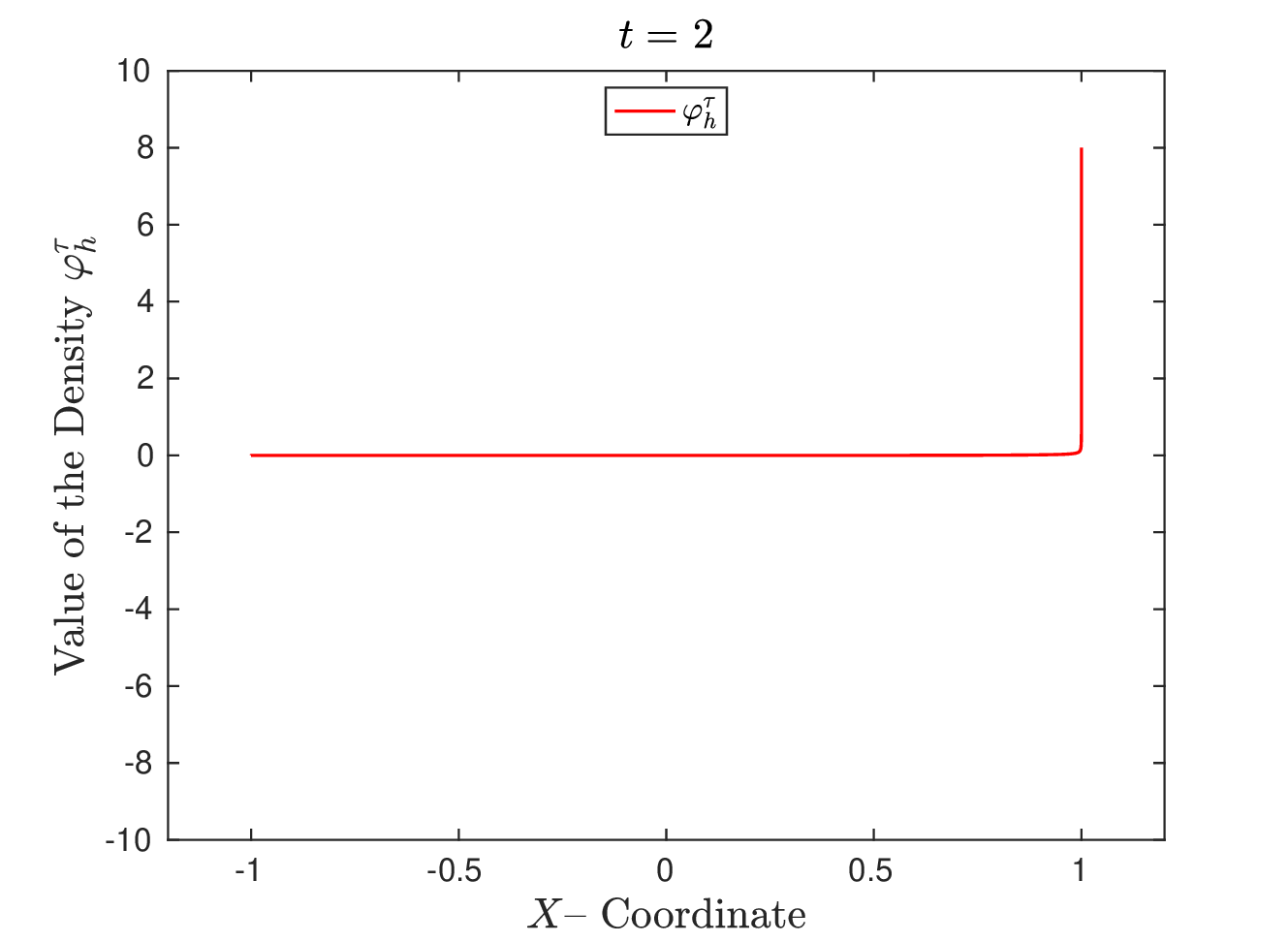}}
    \subfloat[]{
    \includegraphics[width=0.5\linewidth]{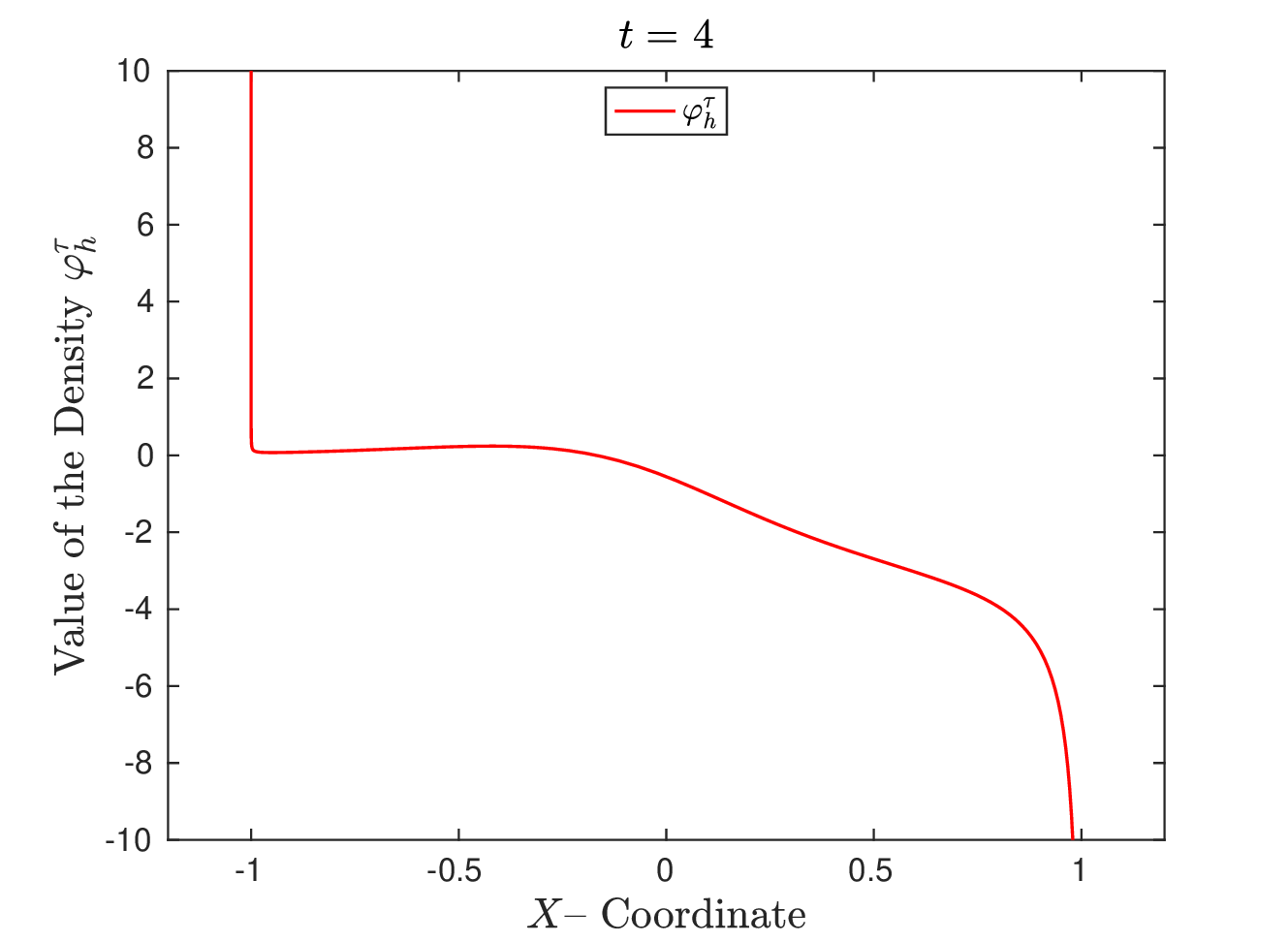}}\\
    \subfloat[]{
    \includegraphics[width=0.5\linewidth]{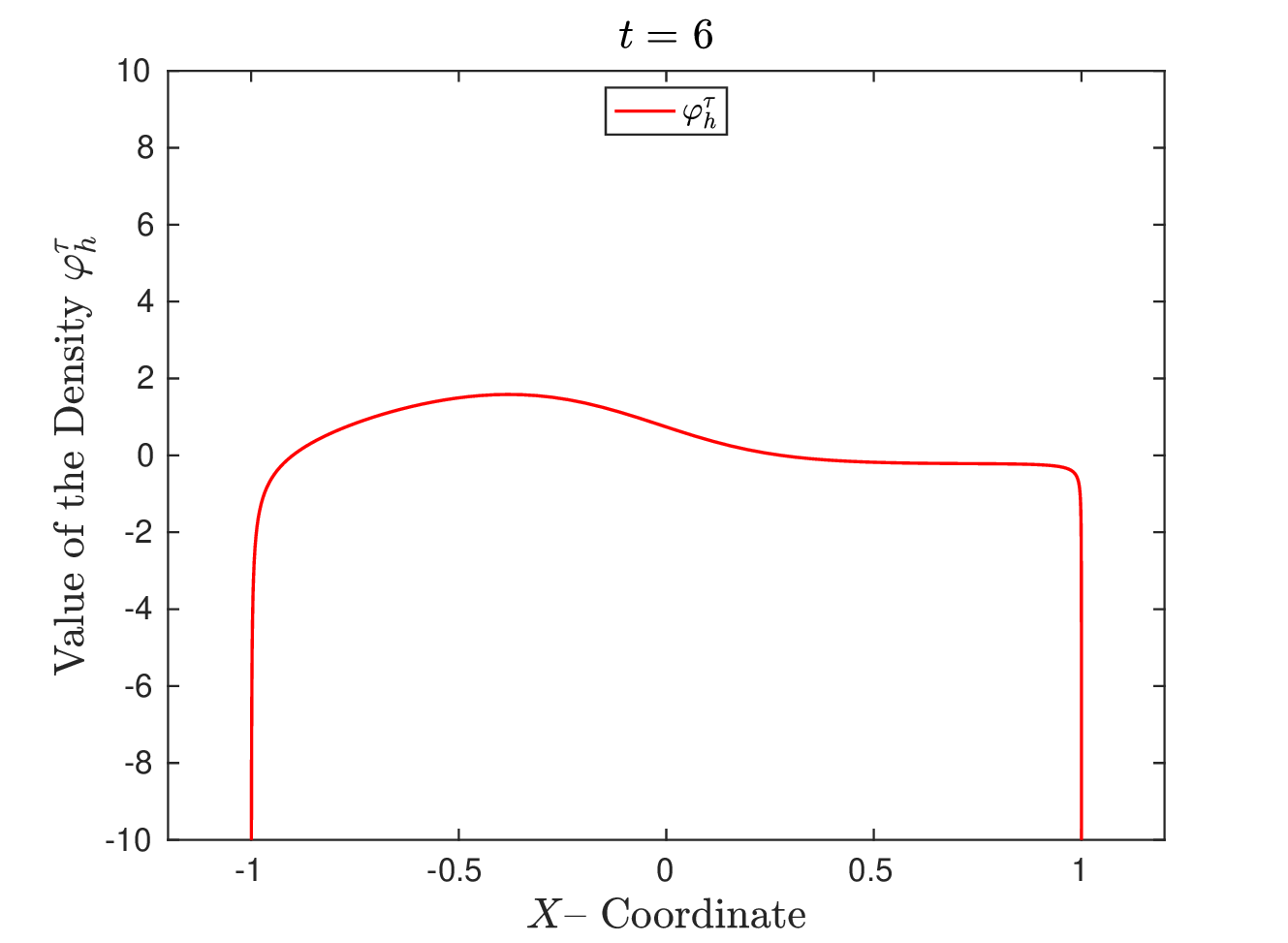}}
    \subfloat[]{
    \includegraphics[width=0.5\linewidth]{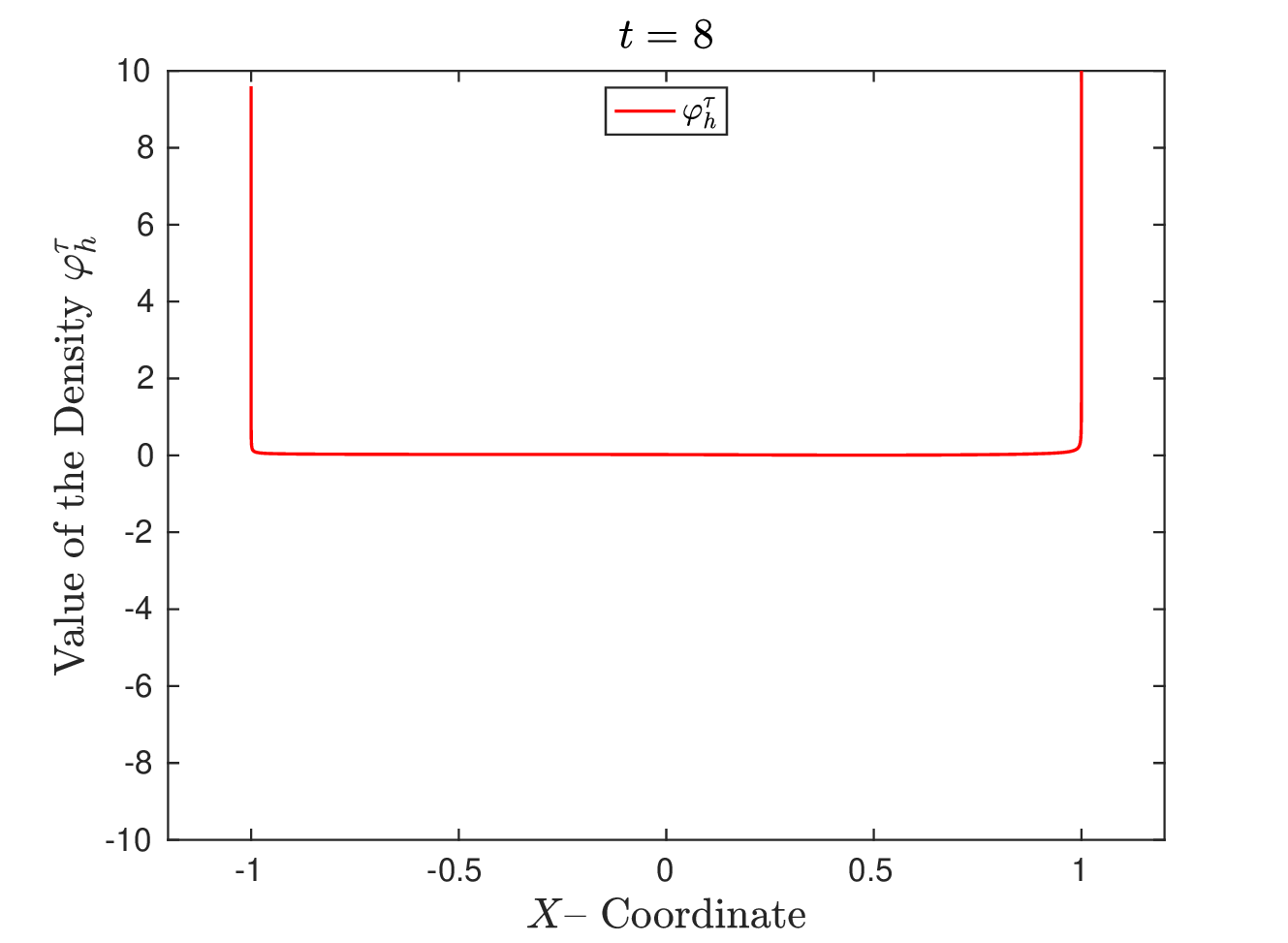}}
    \caption{Snapshots of the density $\varphi_h^{\tau}$: scattering at a flat screen.  }
    \label{fig:snapshots-flat}
\end{figure}

\subsection{Wedge}
We solve the problem at a wedge geometry $\Gamma \coloneqq \Gamma_0 \cup \Gamma_1, $ where $\Gamma_0 \coloneqq (0, 1)\times \{0\}$ and $\Gamma_1 \coloneqq \{0\} \times (0, 1)$. Results for convergence of point evaluation and convergence of densities in energy norm \eqref{eq:energy-norm} are presented in Figure \ref{fig:CQ2}. As before, we observe that our adaptive refinement performs as the optimal $3$--graded mesh refinement. 

\begin{figure}[ht!]
    \centering
    \subfloat[$\tau=0.1$]{
    \includegraphics[width=0.6\linewidth]{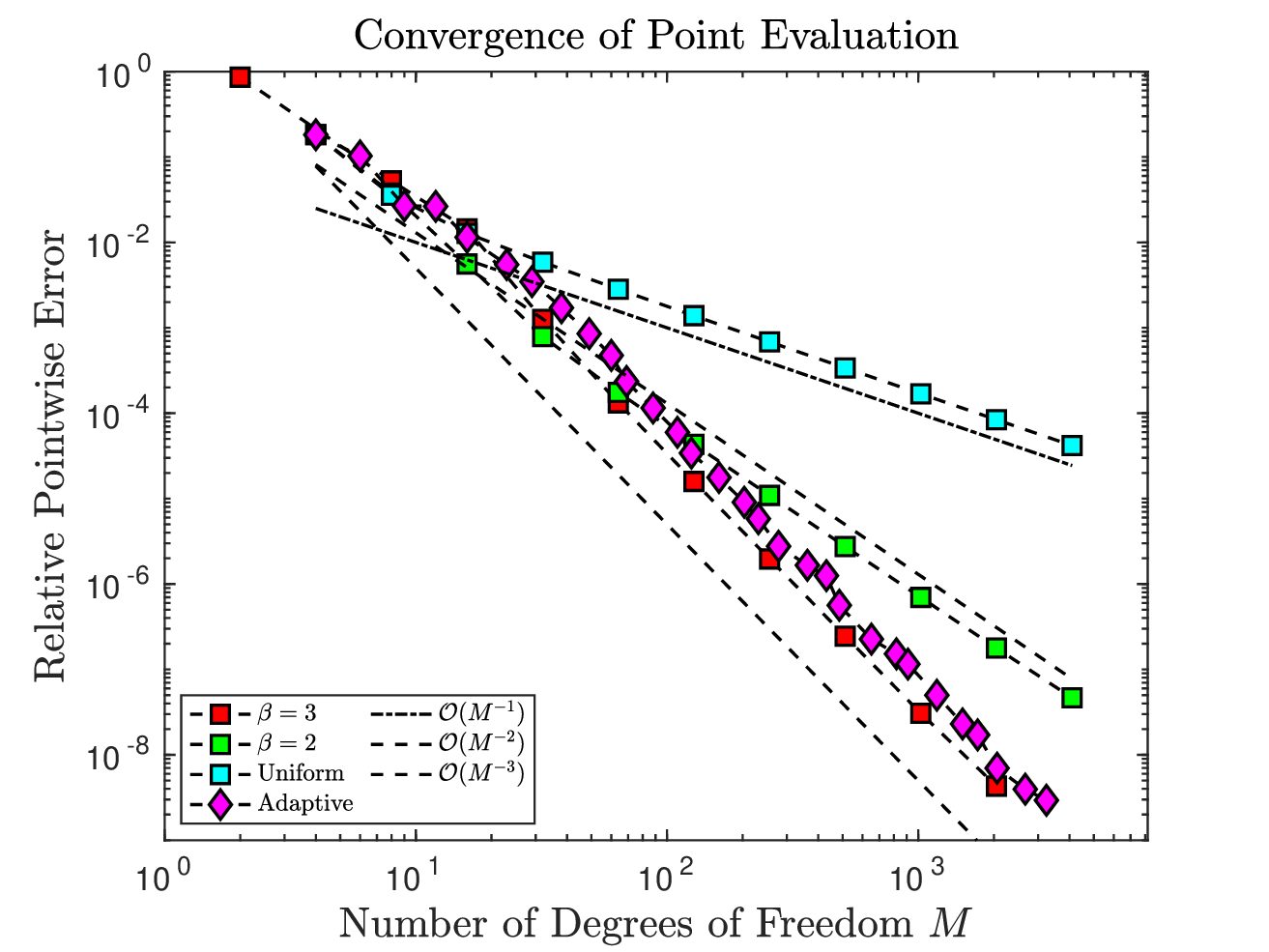}}
    \subfloat[$\tau=0.1$]{
    \includegraphics[width=0.6\linewidth]{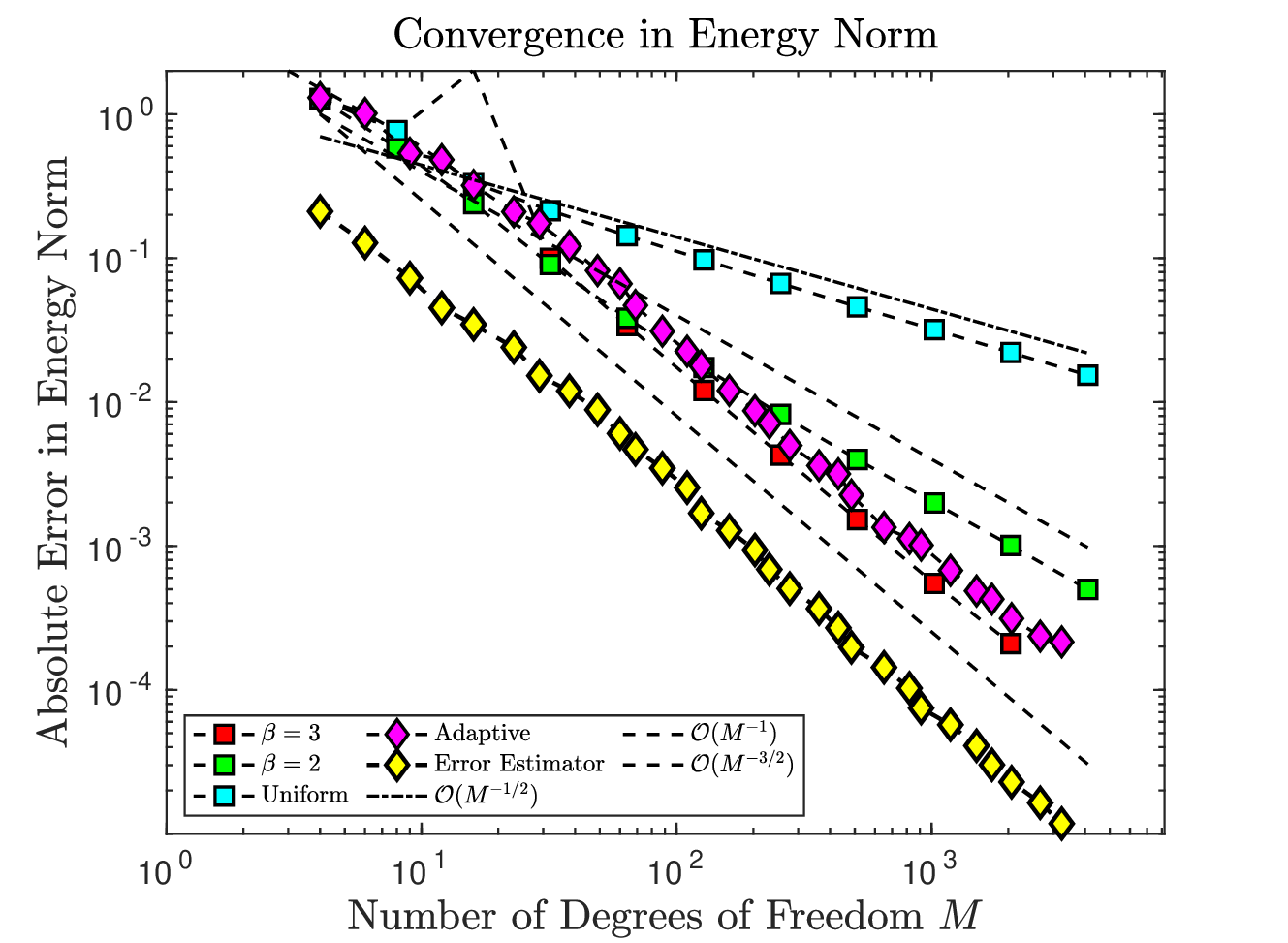}}
    \caption{Scattering at a wedge, $d = (-\tfrac{\sqrt{3}}{2}, \tfrac{1}{2})$. }
    \label{fig:CQ2}
\end{figure}

\subsection{Trapping Geometry}

We solve the problem at a screen with the trapping geometry $$\Gamma \coloneqq \Gamma_0 \cup \Gamma_1\cup \Gamma_2, $$ where each of the segments is defined by
$$
\begin{aligned}
&\Gamma_0 \coloneqq  \{0\} \times (-1,1), \\
&\Gamma_1 \coloneqq \{ \mathbf{x}\in \mathbb{R}^2 : \mathbf{x} = \lambda\mathbf{x}_{10} + (1-\lambda)\mathbf{x}_{11}, \ \lambda \in [0,1] \},\\
&\Gamma_2 \coloneqq \{ \mathbf{x}\in \mathbb{R}^2 : \mathbf{x} = \lambda\mathbf{x}_{20} + (1-\lambda)\mathbf{x}_{21}, \ \lambda \in [0,1] \},
\end{aligned}
$$ 
with $\mathbf{x}_{10} = (0, 1), \ \mathbf{x}_{11} = (\tfrac{\sqrt{3}}{2}, \tfrac{1}{2}), \ \mathbf{x}_{20} = (0, -1),$ and $ \ \mathbf{x}_{21} = (\tfrac{\sqrt{3}}{2}, -\tfrac{1}{2})$. The geometry is visualized in Figure~\ref{fig:snapshots}.

Results for convergence of point evaluation and convergence of densities in energy norm \eqref{eq:energy-norm} are presented in Figure \ref{fig:CQ3}. The additional complexity of the solution induced by the trapping does not seem to affect the performance of the scheme, as the adaptive method still converges with the optimal rate. Snapshots of the solution at times $t = 1, 4, 7$ and $10$ are visualized in Figure \ref{fig:snapshots}.

\begin{figure}[ht!]
    \centering
    \subfloat[$\tau=0.1$]{
    \includegraphics[width=0.6\linewidth]{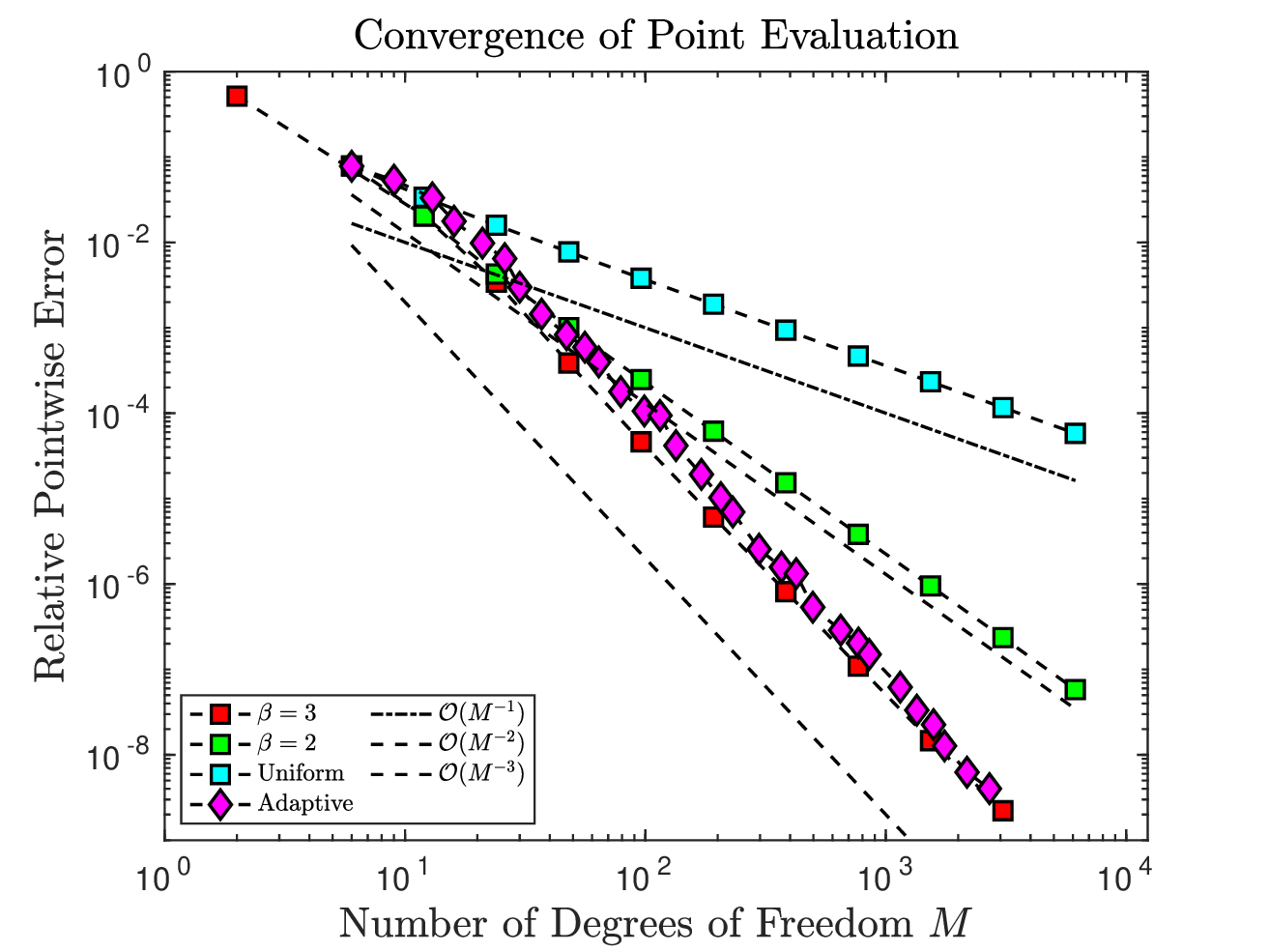}}
    \subfloat[$\tau=0.1$]{
    \includegraphics[width=0.6\linewidth]{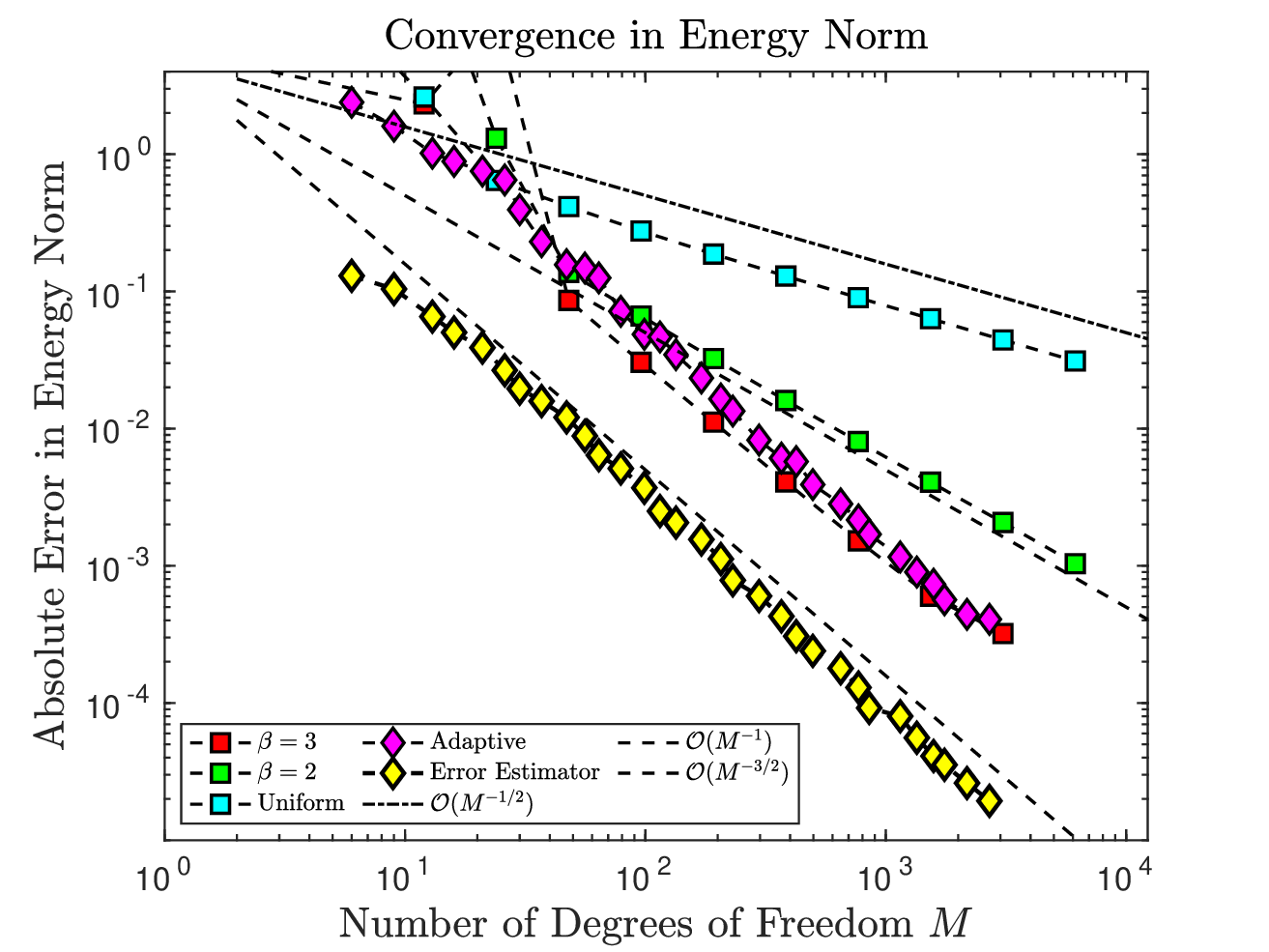}}
    \caption{Scattering at a screen with trapping geometry, $d = (-\tfrac{\sqrt{3}}{2}, \tfrac{1}{2})$. }
    \label{fig:CQ3}
\end{figure}

\begin{figure}[ht!]
    \centering
    \subfloat[]{
    \includegraphics[width=0.5\linewidth]{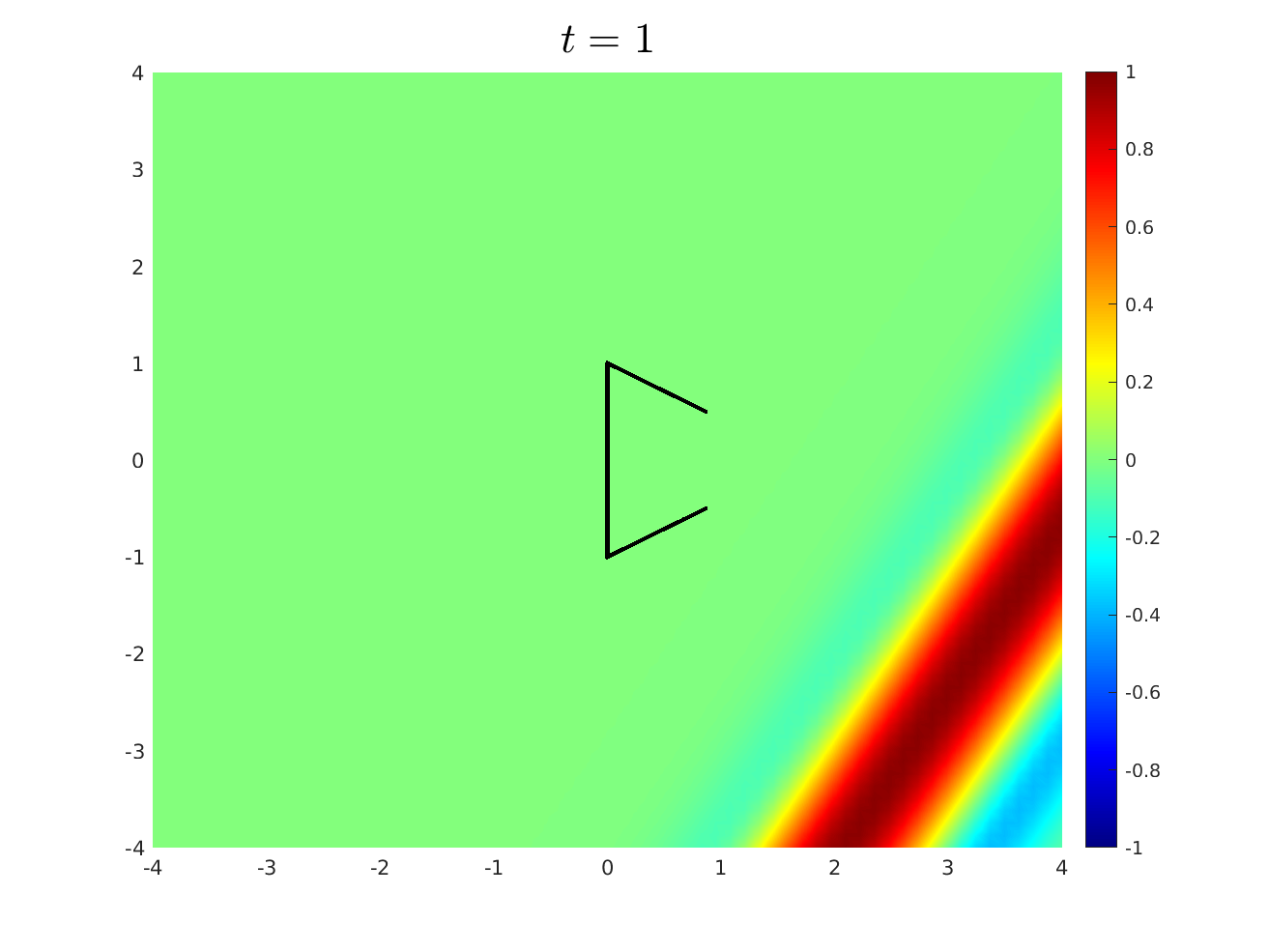}}
    \subfloat[]{
    \includegraphics[width=0.5\linewidth]{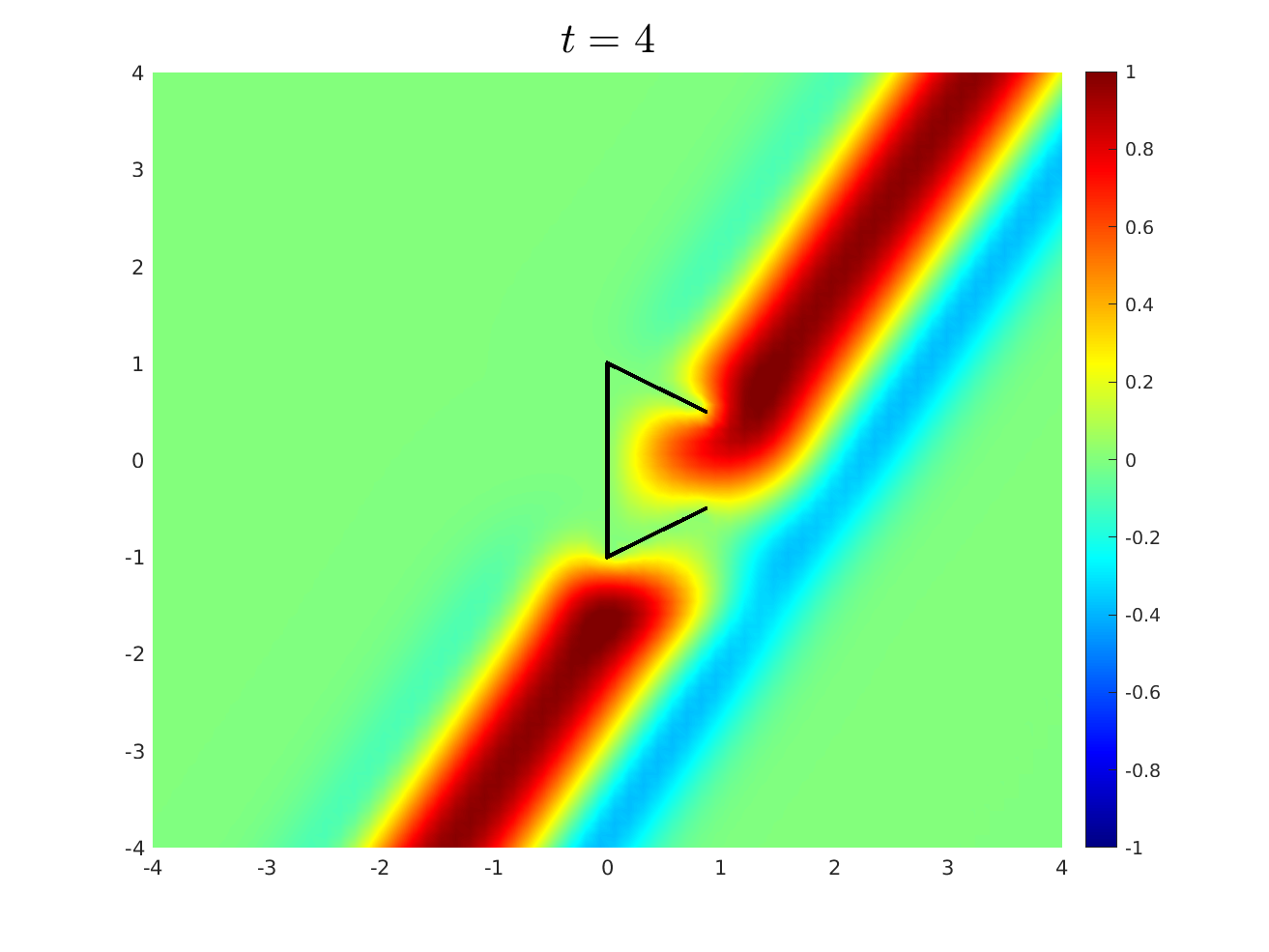}}\\
    \subfloat[]{
    \includegraphics[width=0.5\linewidth]{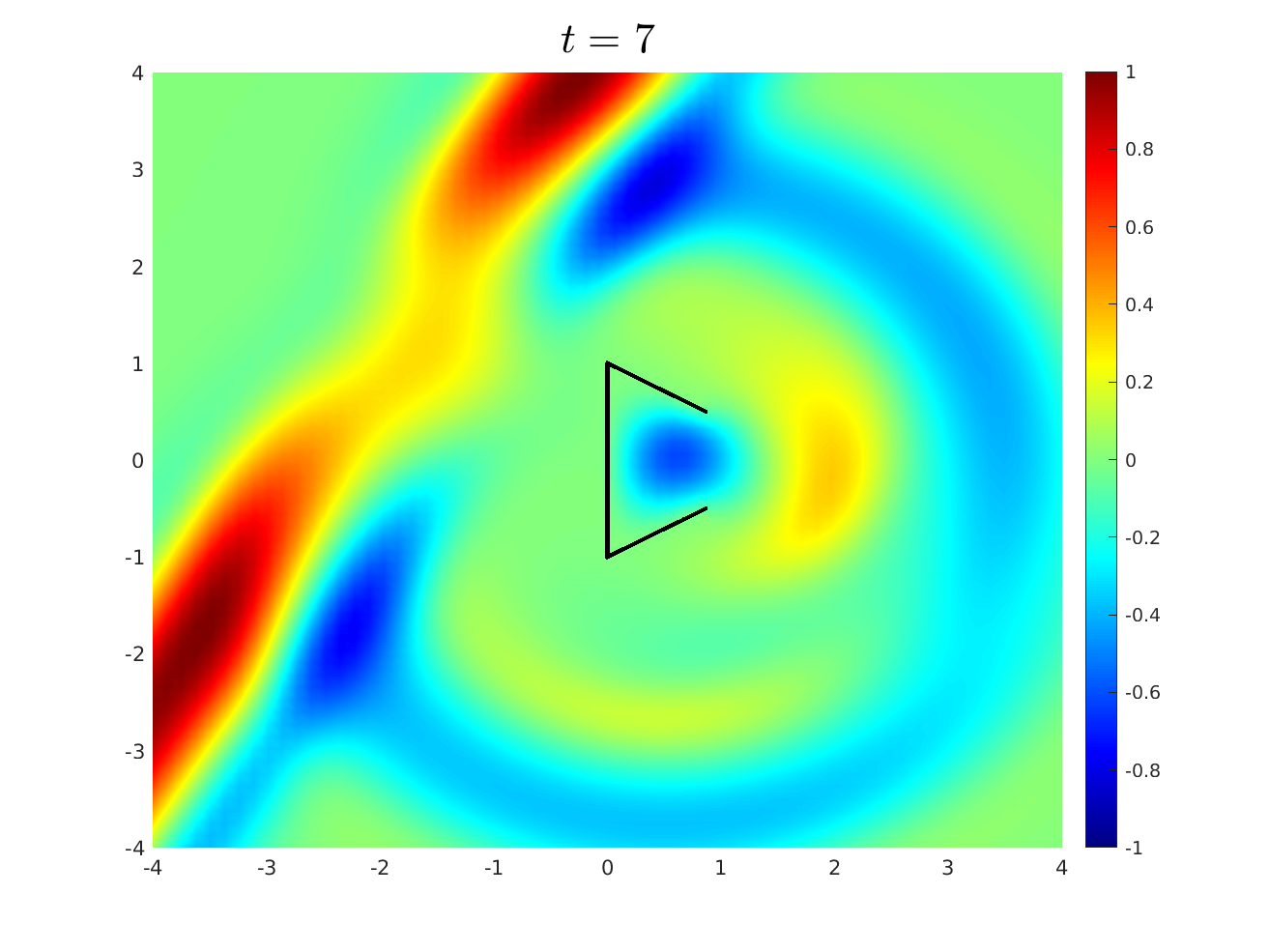}}
    \subfloat[]{
    \includegraphics[width=0.5\linewidth]{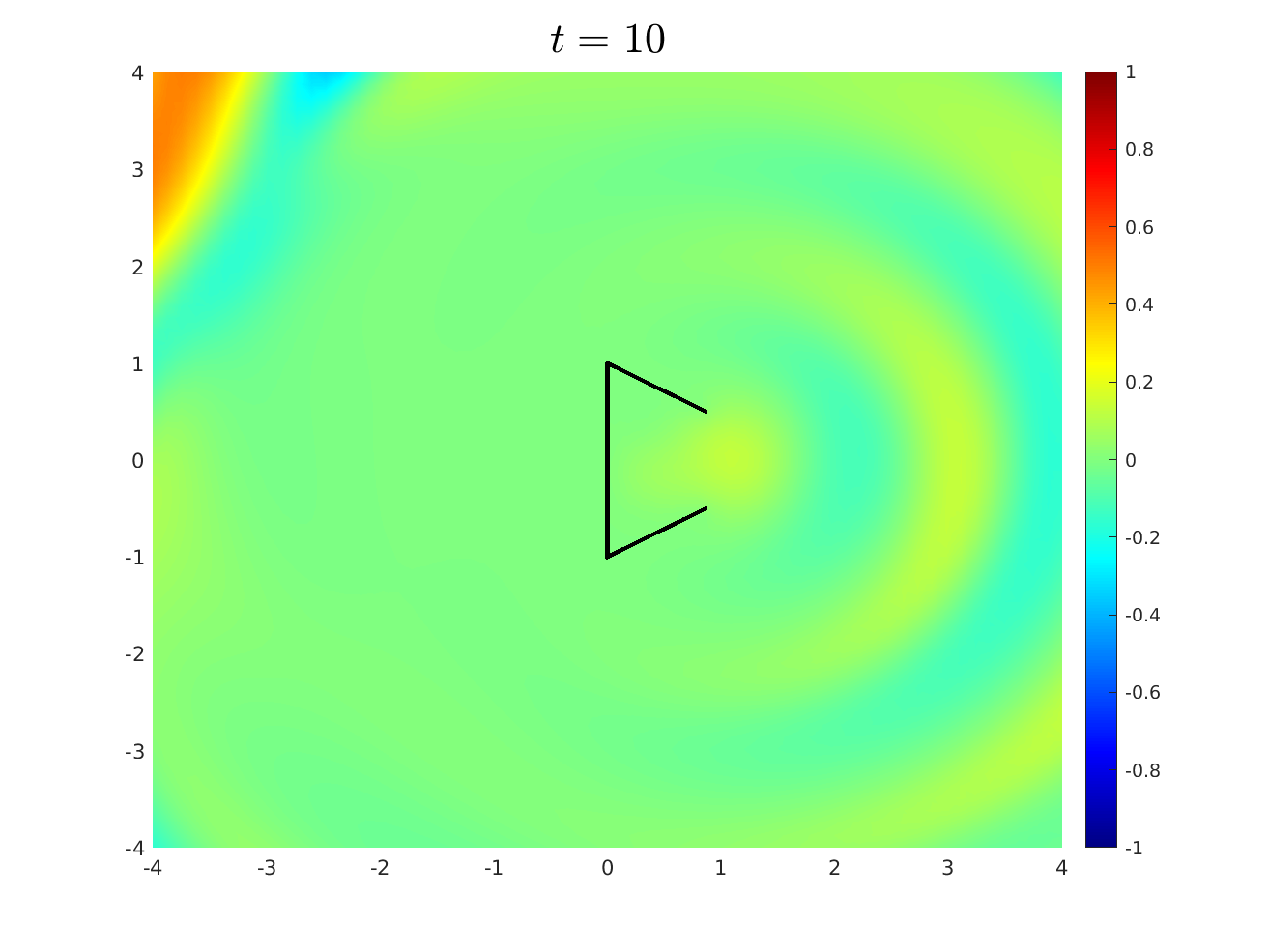}}
    \caption{Snapshots of the solution $u_h^{\tau}$: scattering at a screen with trapping geometry. }
    \label{fig:snapshots}
\end{figure}

\subsection{Time discretization error}
In this section we study the convergence of the time-discretization. We consider the setting of a flat screen from Section \ref{sec:segment} refining with the adaptive algorithm in space. The results are shown in Figure \ref{fig:segment-time}. As a benchmark we use the solution with $\tau = \tfrac{T}{N} =\tfrac{10}{6400} = 0.0015625$.

We observe that the standard formulation ($\eta \equiv 0$) leads to a reduced convergence rate of 2 for the density $\varphi_h^{\tau}$, due to the results from Lemma \ref{lem:RK-CQ}. The modified formulation from \eqref{eq:time-discr-shift} with $\eta = \tfrac{1}{2T}$ shows the classical order of convergence $2m-1 = 3$ for the two-stage Radau IIA method, as Theorem \ref{th:full} predicts. For both cases, point evaluations show the full classical order of convergence.\\

\begin{figure}[ht!]
    \centering
    \includegraphics[width=0.7\linewidth]{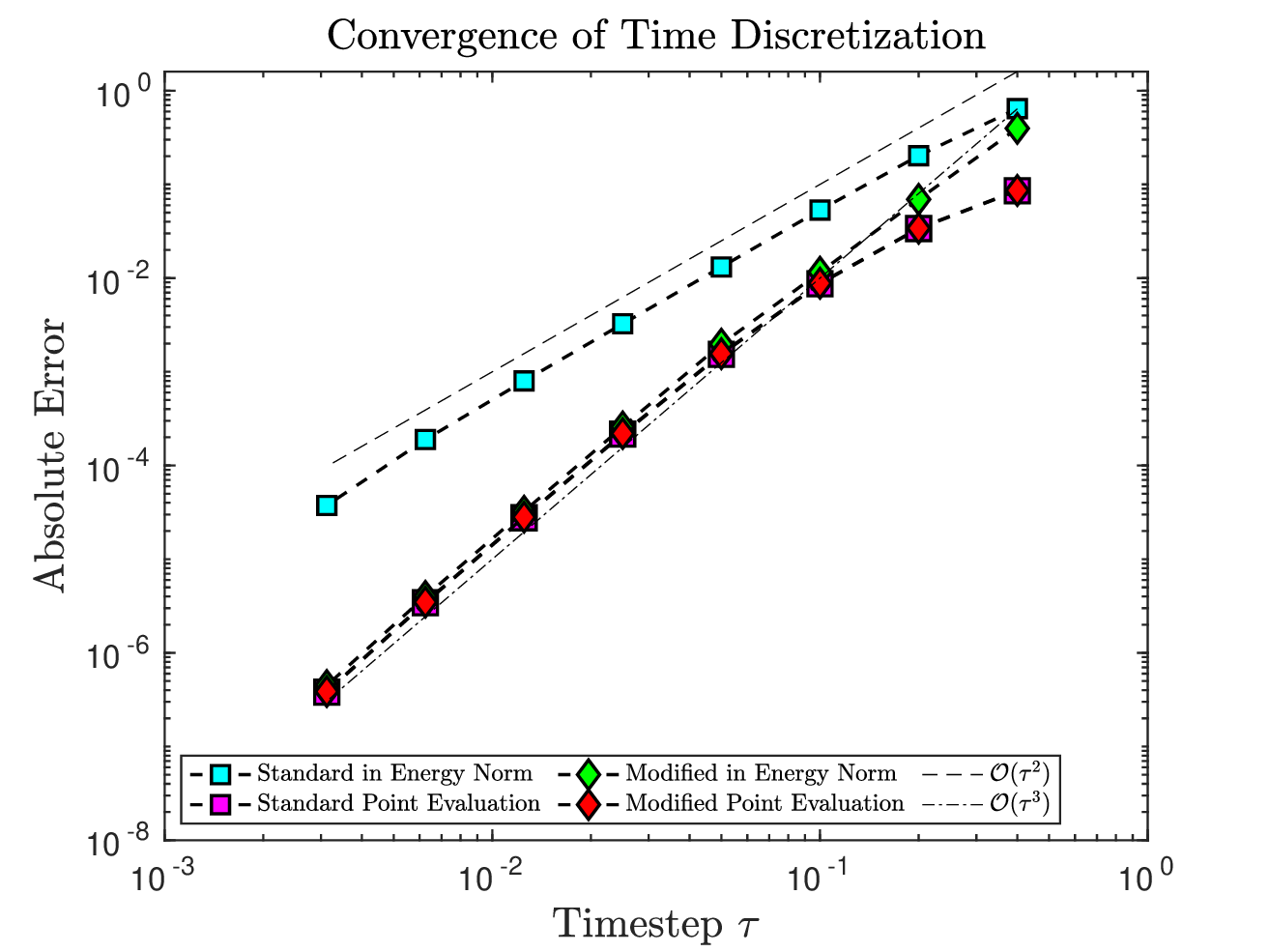}
    \caption{Convergence in time of the standard and modified formulations.}
    \label{fig:segment-time}
\end{figure}

Finally, we investigate the efficiency of the a posteriori error estimate, i.e. the ratio $\mathrm{error}/\mathrm{estimator}$ in Figure \ref{fig:segment-time-error}. The error (and estimator) for $M$ spatial degrees of freedom is written as $e_M$ ($\eta_M$). The computed ratios $e_M/(\eta_M + C \tau^2)$ remain bounded, both as $M\rightarrow \infty$ and the $\tau \rightarrow 0$. We conclude therefore that the efficiency constant remains bounded, even though this cannot be shown for multistage methods, with the current theory. For the constant in the ratio that estimates the time-discretization error, we used $C = 5.$ \\

Similar results are obtained for the modified formulation with a shift of $\tfrac{1}{2T}$. The results are shown in Figure \ref{fig:segment-time-error3}. Here, we empirically observe a convergence rate of 2.8, close to the classical convergence rate of 3. We plot $e_M/(\eta_M + C \tau^{2.8})$ with a constant $C = 9$, and observe how most of the values remain around 1. It is important to mention that the mesh obtained by the adaptive algorithm contains elements of very small size towards the endpoints of the segment. This leads to ill-conditioned matrices during the computation of the reference solution. Further experiments require the use of preconditioning, and a higher number of quadrature points.

\begin{figure}[ht!]
    \centering
    \subfloat[Convergence of adaptive scheme.]{
    \includegraphics[width=0.6\linewidth]{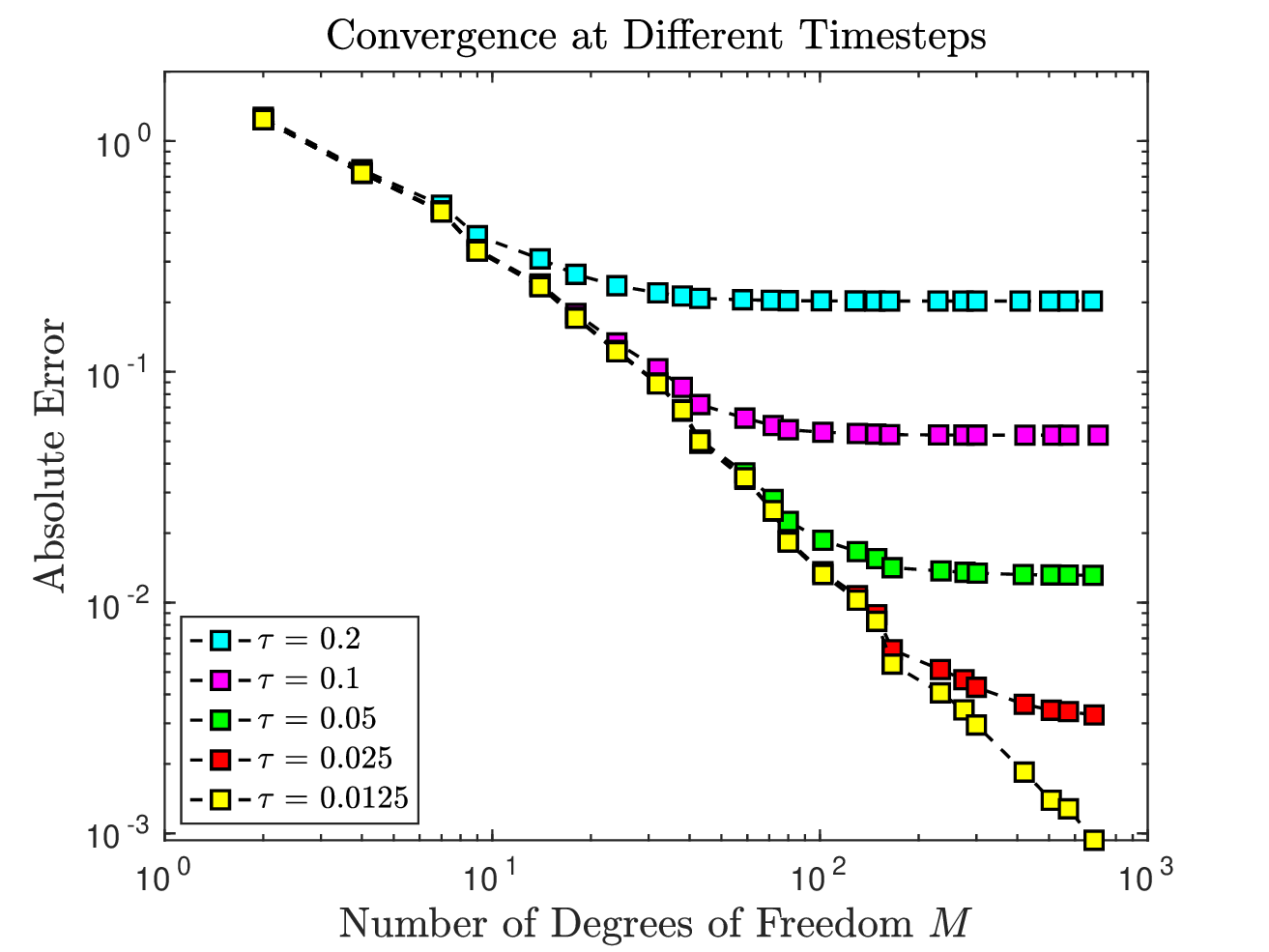}}
    \subfloat[Computing the efficiency of error estimator]{ \includegraphics[width=0.6\linewidth]{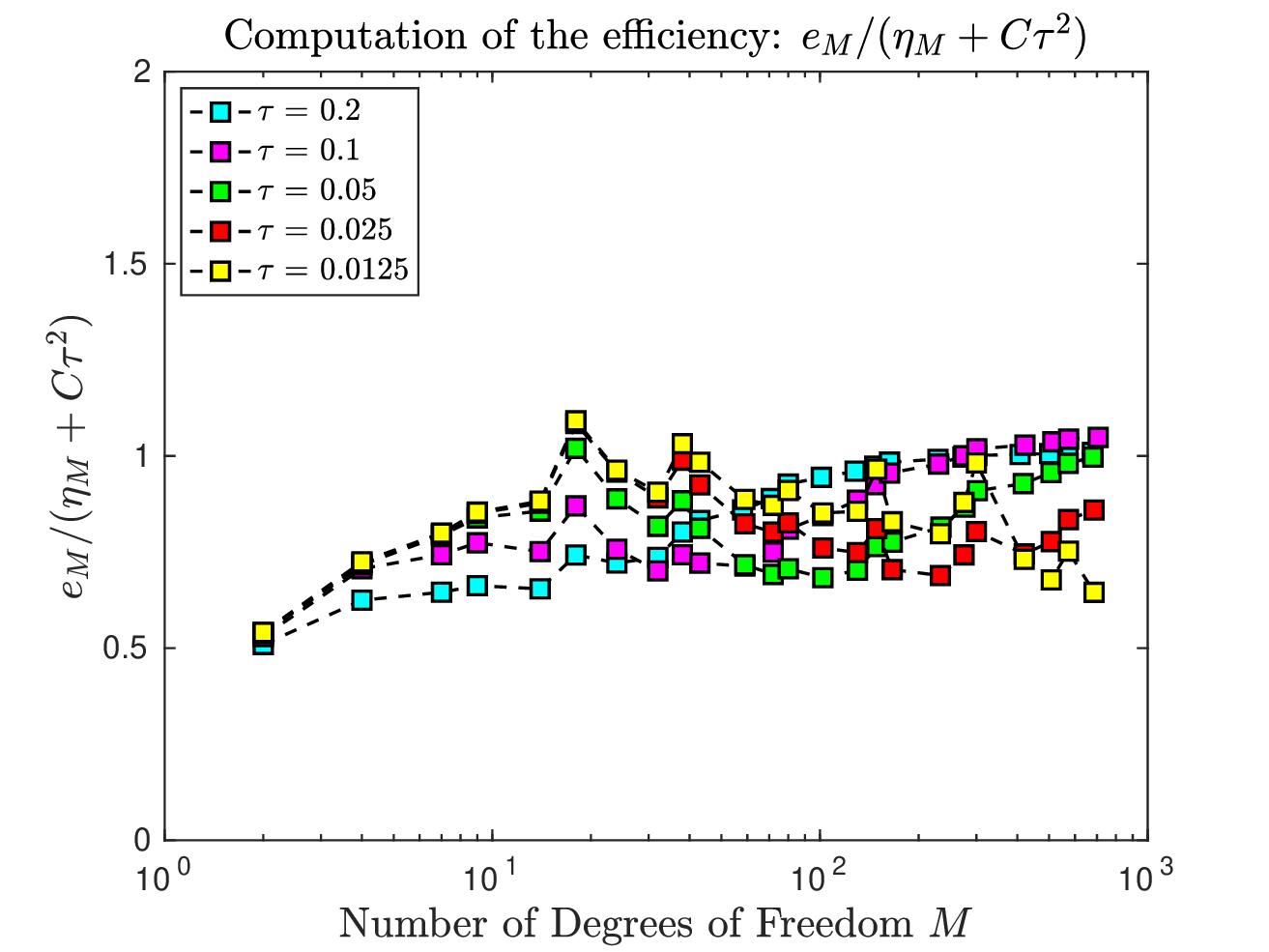}}
    \caption{Error and efficiency for different time steps. Standard formulation.}
    \label{fig:segment-time-error}
\end{figure}

\begin{figure}[ht!]
    \centering
    \subfloat[Convergence of adaptive scheme.]{
    \includegraphics[width=0.6\linewidth]{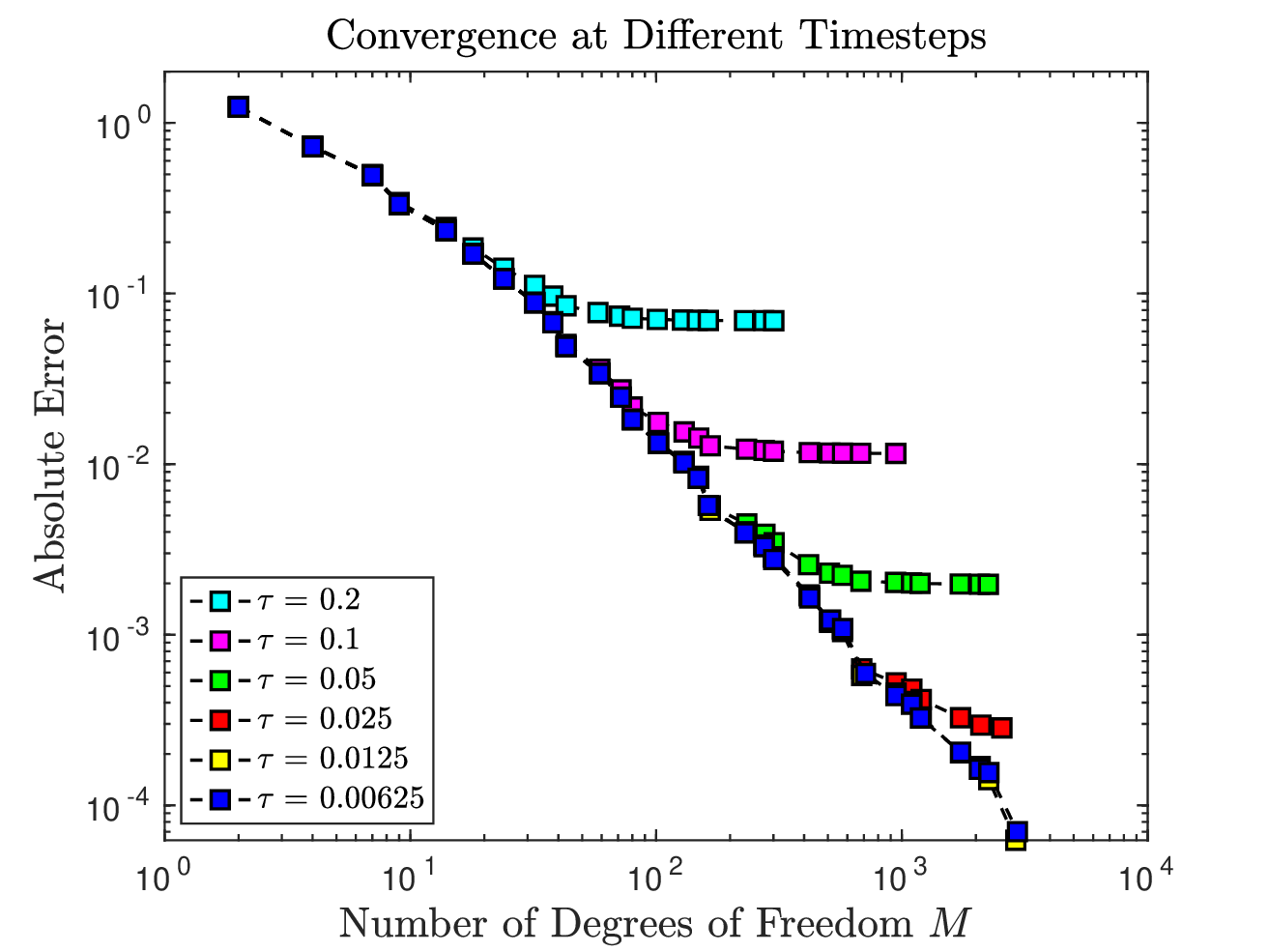}}
    \subfloat[Computing the efficiency of error estimator]{ \includegraphics[width=0.6\linewidth]{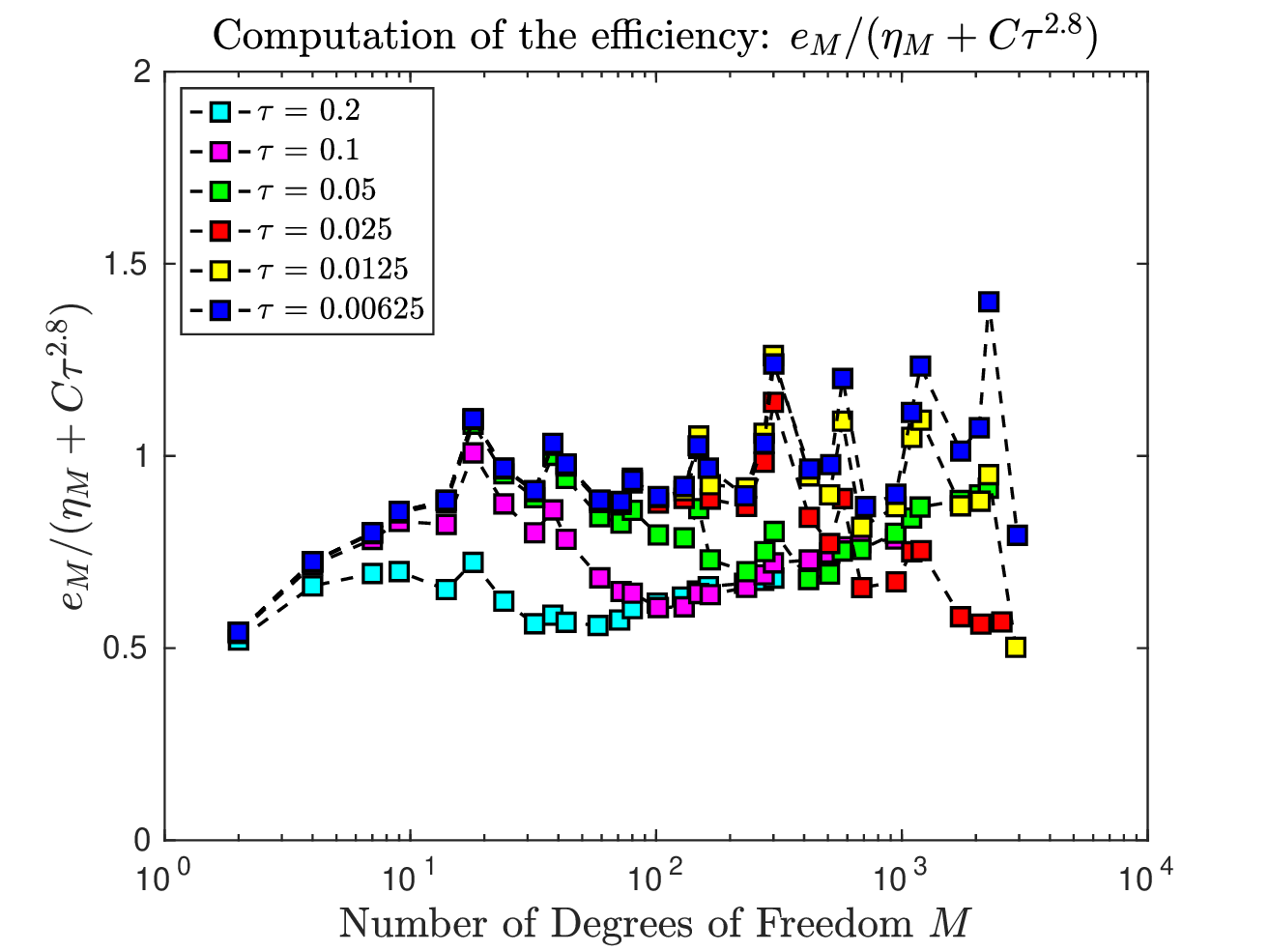}}
    \caption{Error and efficiency for different time steps. Modified formulation.}
    \label{fig:segment-time-error3}
\end{figure}

\section*{Acknowledgements}

We thank C.~Schwab for initiating early discussions, which led to the inception of this project.

	\bibliographystyle{abbrv}
	\bibliography{Lit}
\end{document}